\documentclass{article}
\usepackage{graphicx}%
\usepackage{multirow}%
\usepackage{setspace}
\usepackage[numbers,sort&compress]{natbib}
\usepackage{amsmath,amssymb,amsfonts}%
\usepackage[english]{babel}
\usepackage{amsthm}%
\usepackage{mathrsfs}%
\usepackage[title]{appendix}%
\usepackage{xcolor}%
\usepackage{enumitem}
\usepackage[font=footnotesize,labelfont=bf,indent]{caption} 
\usepackage{booktabs}%
\usepackage{bm} 

\newcommand{\fp}{\mathscr{U}_\gamma}
\newcommand{\sample}{\mathscr{S}_\gamma}
\newcommand{\neff}{n_{\text{eff},\gamma}}
\newcommand{\tneff}{n_{\text{V-eff},\gamma}}
\newcommand{\tr}{^\intercal}
\newcommand{\mat}{\mathbf}
\newcommand{\argmin}{\operatorname*{arg\,min}}
\newcommand{\argmax}{\operatorname*{arg\,max}}

\renewcommand{\Pr}{\mathbb{P}}
\newcommand{\dx}{\,\mathrm{d}}

\newcommand{\R}{\mathbb{R}}
\newcommand{\E}{\mathbb{E}}
\newcommand{\Pp}{\mathbb{P}}
\newcommand{\1}{\mathbf{1}}

\newcommand{\Var}{\mathrm{Var}}
\newcommand{\Mid}{\,\big|\,}
\newcommand{\MID}{\,\Big|\,}

\newtheorem{theorem}{Theorem}[section]
\newtheorem{lemma}[theorem]{Lemma}
\newtheorem{proposition}[theorem]{Proposition}
\newtheorem{corollary}[theorem]{Corollary}

\theoremstyle{remark}
\newtheorem{remark}[theorem]{Remark}

\def\spacingset#1{\renewcommand{\baselinestretch}%
    {#1}\small\normalsize} \spacingset{1}

\begin{document}

\title{Minimum Hellinger Distance Estimators for Complex Survey Designs}

\author{David Kepplinger and Anand N. Vidyashankar}
\date{Department of Statistics, George Mason University, Fairfax, VA, USA}

\maketitle

\begin{abstract}
Reliable inference from complex survey samples can be derailed by outliers and high-leverage observations induced by unequal inclusion probabilities and calibration. We develop a minimum Hellinger distance estimator (MHDE) for parametric superpopulation models under complex designs, including Poisson PPS and fixed-size SRS/PPS without replacement, with possibly stochastic post-stratified or calibrated weights. Using a Horvitz--Thompson–adjusted kernel density plug-in, we show: (i) $L^1$–consistency of the KDE with explicit large-deviation tail bounds driven by a variance-adaptive effective sample size; (ii) uniform exponential bounds for the Hellinger affinity that yield MHDE consistency under mild identifiability; (iii) an asymptotic Normal distribution for the MHDE with covariance $\mat A^{-1}\bm\Sigma \mat A\tr$ (and a finite-population correction under without-replacement designs); and (iv) robustness via the influence function and $\alpha$–influence curves in the Hellinger topology. Simulations under Gamma and lognormal superpopulation models quantify efficiency--robustness trade-offs relative to weighted MLE under independent and high-leverage contamination. An application to NHANES 2021--2023 total water consumption shows that the MHDE remains stable despite extreme responses that markedly bias the MLE. The estimator is simple to implement via quadrature over a fixed grid and is extensible to other divergence families.%
\end{abstract}

\noindent%

\newpage
\spacingset{1.5} 

\section{Introduction}

Reliable estimation and inference in complex survey samples is a challenging problem, particularly when outliers can be present.
In this work, we develop a robust estimator and asymptotic inference for survey samples from a finite population with possibly unequal inclusion probabilities which can be based on auxiliary information.
Outliers, i.e., unusually small or large values, in the observed sample must be handled carefully to avoid biased and invalid inference from the sample survey.
These outliers may be legitimate values, but can also be caused by data entry errors and other problems.
Regardless of the legitimacy of these unusual values, inclusion probabilities from auxiliary information can drastically amplify the outliers' effects on the estimator.
Borrowing from the terminology of high breakdown estimators for linear regression, we call outliers in units with low inclusion probability \textit{high-leverage observations}.

In large-scale surveys, it is common to adjust the survey weights derived from the inclusion probabilities, to match certain characteristics to known totals for the entire population or within strata \citep{Zhang2000}.
Post-stratification or calibration leads to stochastic survey weights even if the initial inclusion probabilities are deterministic.
When such adjustments are applied, outliers and high-leverage observations can be further amplified and have an even more detrimental effect on an estimator.

We propose a reliable minimum Hellinger distance estimator (MHDE) for model parameters under complex survey designs, with potentially random survey weights.
Minimum divergence estimators are known for their robustness toward outliers without sacrificing efficiency in clean samples in a wide range of models and settings \citep[e.g.,][]{Beran1977,Donoho1988,Simpson1989,Lindsay1994,Lu2003}.
Recently, minimum phi-divergence estimators have been shown to achieve robustness and high efficiency for multinomial and polytomous logistic regression in complex survey designs \citep{Castilla2018a,Castilla2021}.
Following that line of research, we develop an MHDE for the parameters of a superpopulation model from a survey sample.
We allow inclusion probabilities derived from auxiliary information, e.g., probability proportional to size (PPS) sampling \citep{Sarndal1992}, cluster sampling or stratified sampling, and stochastic survey weights adjusted by post-stratification or calibration.

In Section~\ref{sec:method} we define our MHDE for complex survey designs and show in Section~\ref{sec:theory} that it is consistent under mild assumptions, is asymptotically Normal and robust under arbitrary contamination.
The empirical studies in Section~\ref{sec:empirical-studies} demonstrate that the estimator is highly efficient and yields valid inference, even in the presence of outliers and high-leverage observations.
We further apply our estimator to the National Health and Nutrition Examination Survey (NHANES) \citep{nhanes2025}, where we show that our MHDE is much less affected by unusual values than the maximum likelihood estimator.

\subsection{Background}

For each $\gamma \in \mathbb N$ we consider a finite population of $N_\gamma$ units $\fp=\{1, \dotsc N_\gamma\}$.
We observe i.i.d.\ draws $\{(Y_{\gamma i}, Z_{\gamma i})\colon i \in \fp\}$ from a superpopulation law on $\mathbb R^{d+1} \times (0,\infty)$.
The $Y_{\gamma i}$ are the characteristics of interest with unknown but measurable and integrable density $g \in L^1(\mathbb R^d)$.
The auxiliary variable $Z_{\gamma i}$ can be used to derive inclusion probabilities and, if used, is assumed to be known and greater than 0 with probability 1.
From $\fp$, a sample of size $n_\gamma$ is drawn according to pre-specified, potentially unequal, inclusion probabilities $\pi_{\gamma i} > 0$, $i \in \fp$.
The units included in the random sample are denoted by $\sample \subset \fp$.

For simple designs, such as fixed-size simple random sampling (SRS) with or without replacement, the inclusion probabilities are equal for all units.
However, in many survey samples, the design is more complicated.
In this work, we focus on the probability proportional-to-size (PPS) design with random size (Poisson--PPS), where the inclusion probabilities depend on an auxiliary variable, $Z_{\gamma i}$, $i \in \fp$, with $\pi_{\gamma i} = \frac{n_{\gamma}Z_i}{\sum_{k\in\fp} Z_k}$ and hence $\sum_i \pi_{\gamma i} = n_\gamma$.
In the PPS design, $Z_{\gamma i}$ is known prior to sampling for all units in $\fp$, e.g., the earnings of business entities in the previous year(s) or the taxable income of households.
If the auxiliary variable, $Z_{\gamma i}$, is correlated with $Y_{\gamma i}$, PPS sampling can reduce the sampling variance of an estimator.
Other sampling designs also use auxiliary information to derive inclusion probabilities, such as cluster sampling or stratified sampling based on geographic location or school districts, to name just a few.
Unequal inclusion probabilities yield non-identically distributed observations, as the unit-level density becomes $\tilde g_{\gamma,i}(y) = \pi_{\gamma i} g(y)$.

Based on the inclusion probabilities, we define the sample weight for each unit in the sample as $w_{\gamma i} = 1 / \pi_{\gamma i}$, $i \in \sample$.
These sample weights need to be considered in the estimation to achieve consistency and reduce bias, e.g., using the Horvitz-Thompson (HT) adjustment \citep{Horvitz1952}.
However, with post-stratification or calibration, the sample weights may be further adjusted to $\omega_{\gamma i} = \zeta_{\gamma i} w_{\gamma i}$, where $\zeta_{\gamma i}$ is a positive random factor with $\Pr(\zeta_{\gamma i} > 0) = 1$.

In this paper our goal is to find a parametric density from a family $\mathscr{F} := \{f_\theta\colon \theta \in \Theta \subseteq \mathbb{R}^p\}$ which is ``closest'' to the true superpopulation distribution in the topology defined by divergence $D$.
Hence, we seek $f_{\hat\theta}$ where $\hat\theta_\gamma = \argmin_{\theta\in\Theta} D(f_\theta \| g)$.

The theoretical and empirical properties of $\hat\theta$ are intricately linked to the divergence, $D$.
Information divergences between probability density functions are a rich family of measures, but not all are suitable under model misspecification, i.e., $g \notin \mathscr{F}$, for example the Tukey-Huber $\varepsilon$-contamination model \citet{Tukey1959}.
The Kullback-Leibler divergence, for example, yields the maximum likelihood estimate \citep{Cover2001} but can lead to arbitrarily biased estimates under contamination.
In this paper, we therefore focus on the more robust (squared) Hellinger distance:
\begin{equation}
H^2(f, g) = 
    \frac{1}{2} \int_\Omega \left(
      \sqrt{f(y)} - \sqrt{g(y)}
    \right)^2
    \dx y =
    1 - \int_\Omega \sqrt{f(y) g(y)} \dx y.
\end{equation}

The Hellinger divergence is known to yield estimates that are robust towards model misspecification \citep{Beran1977}, yet achieve high efficiency if $g \in \mathscr{F}$ \citep{Lindsay1994}.
Important for large-scale surveys, the MHDE can be quickly computed using numerical integration if the dimension $Y$ is reasonably low.
In the following Section~\ref{sec:method} we describe the MHDE for complex survey designs based on a Horvitz-Thompson adjusted Kernel Density Estimator.

%


\subsection{Notation}

Throughout we denote by $\delta_{\gamma i} \in \{0, 1\}$ whether a unit in the finite population $\fp$ is included in the sample $\sample$ or not, i.e., $\delta_{\gamma i} = 1$ if $i \in \sample$ and $\delta_{\gamma i} = 0$ otherwise.
The first-order inclusion probabilities are thus $\pi_{\gamma i} = \Pr(\delta_{\gamma i} = 1)$.
We let the ``effective'' sample size be $\neff := N_\gamma^2 / \sum_{i\in\fp} \pi_{\gamma,i}^{-1}$, and the variance-adaptive effective sample size $\tneff := N_\gamma^2 / \sum_{i\in\fp} (1 - \pi_{\gamma,i}) / \pi_{\gamma,i}$.
For fixed-size designs, such as SRS--WOR or fixed-size PPS--WOR, we write $\alpha_\gamma := n_\gamma / N_\gamma$.
The bandwidth of the kernel density estimator depends on the sample size $n_\gamma$ and we denote it by $h_\gamma > 0$.
We then write the normalized kernel as $K_{h_\gamma}(x) = h_\gamma^{-d} K(x / h_\gamma)$.

We denote the Hellinger affinity (Bhattacharyya coefficient) between the parametric density $f_\theta$ and the KDE $\hat f_\gamma$ or the (arbitrary) distribution $F$ with density $f$, respectively, by
\begin{align*}
\Gamma_\gamma(\theta) := \int \sqrt{\hat f_\gamma(y) f_\theta(y)} \dx y, &&
\Gamma_F(\theta) := \int \sqrt{f(y) f_\theta(y)} \dx y.
\end{align*}
We simply write $\Gamma(\theta) := \Gamma_{G}(\theta)$ when referring to the true superpopulation distribution.

Finally, we define the score function as $u_\theta(y) := \nabla_\theta \log f_\theta(y)$ and use
\begin{align*}
\phi_g(y) := \frac{1}{4} u_{\theta_0}(y) \sqrt{\tfrac{f_{\theta_0}(y)}{g(y)}}, &&
\bm\Sigma := \E_{G} \left[ \phi_{g}(Y) \phi_{g}(Y)\tr \right], &&
\mat A := -\nabla^2_\theta\, \Gamma(\theta) \Big|_{\theta = \theta_0},
\end{align*}
to denote the scaled score function, the expected information and the Hessian of the Hellinger affinity, respectively.
Where obvious, we omit the subscript from the scaled score function and write $\phi(y) := \phi_g(y)$.

\section{Methodology}\label{sec:method}

Let the true superpopulation distribution of $Y$ be $G$ with density $g$.
Introducing the parametric family $\mathscr{F}=\{f_{\theta}: \theta\in\Theta\}$, the population minimizer $\theta_0 = \argmin_{\theta \in \Theta} H^2(f_\theta, g)$ represents the ``closest'' parametric density in $\mathscr F$ to $G$ in the Hellinger topology.
To estimate $\theta_0$, we minimize the Hellinger distance between the estimated density and the densities in the parametric family:
\begin{equation}\label{eqn:mhde-def}
\hat\theta_\gamma := \argmin_{\theta \in \Theta} H^2(f_{\theta}, \hat f_\gamma) =
\argmax_{\theta \in \Theta} \Gamma_\gamma(\theta).
\end{equation}

To obtain a consistent estimate of $\theta_0$, we use the Horvitz-Thompson (HT) adjusted kernel density estimator:
\begin{equation}
\hat f_\gamma(y):=\frac{1}{\sum_{i\in\fp} \delta_{\gamma i}/\pi_{\gamma i}} \sum_{i\in\fp} \frac{\delta_{\gamma i}}{\pi_{\gamma i}} K_{h_\gamma}(Y_{\gamma i} - y).
\end{equation}

Underpinning the robustness properties of the MHDE defined in~\eqref{eqn:mhde-def} is its continuity in the Hellinger topology, as shown in Proposition~\ref{prop:continuity} in the Appendix under mild conditions.
As the proportion of contaminated observations decreases, the estimator converges to the maximizer of~\eqref{eqn:mhde-def} without contamination.

\subsection{A Note on Computation}

Our software implementation maximizes $\Gamma_\gamma(\theta)$ by Nelder-Mead \citep{Nelder1965}.
The integral in $\Gamma_\gamma$ is computed over a fixed grid $\mathcal G_Y$ using the Gauss-Kronrod quadrature and a given number of subdivisions.
Therefore, $\hat f_\gamma$ must be evaluated only once for each $y \in \mathcal G_Y$.
Particularly for large $n_\gamma$, this substantially eases the computational burden compared to adaptive quadrature.
The grid is chosen to cover only the regions where $\hat f_\gamma > 0$, which can be quickly evaluated knowing the kernel and bandwidth.

\section{Theory}\label{sec:theory}

We present three main results for the MHDE~\eqref{eqn:mhde-def} in the finite population setting under the superpopulation model framework.
We first show that the HT-adjusted KDE under PPS sampling converges in $L_1$ to the superpopulation density $f_\gamma$, while the naïve KDE converges to a size-biased density.
We then prove that the MHDE based on the HT-adjusted KDE is consistent for $\theta_0$ and derive its limiting normal distribution under several sample designs.
Finally, we obtain the influence function and demonstrate the robustness of the estimator.

In the following, we write the HT-adjusted KDE as
$
\hat f_{\gamma}(y) = \frac{1}{S_\gamma} T_\gamma(y),
$
with
\begin{align*}
S_\gamma:=\frac{1}{N_\gamma}\sum_{i\in\fp}\frac{\delta_{\gamma i}}{\pi_{\gamma i}}, &&
T_\gamma(y):=\frac{1}{N_\gamma}\sum_{i\in\fp}\frac{\delta_{\gamma i}}{\pi_{\gamma i}}K_{h_\gamma}(y-Y_{\gamma i}).
\end{align*}

\subsection{Consistency of the HT-adjusted KDE}

For consistency to hold, we assume that the kernel function $K$ is smooth and that the bandwidth decreases at a prescribed rate.
We also make concrete our assumptions about the superpopulation model and the regularity of the design.

\begin{enumerate}[label=\textbf{A\arabic*}, resume=assumptions]
  \item (Smoothness of the kernel).
  The kernel $K \in L^1 \cap L^2$ is bounded, non-negative, Lipschitz ($\nabla K \in L^1$) and integrates to one, $\int_{\mathbb R^d} K(x) \dx x = 1$.
  \label{ass:kernel-1}
  \item (Bandwidth and growth).
  The bandwidth $h_\gamma\to0$ such that $\neff h_\gamma^d \to \infty$ and $N_\gamma h_\gamma^d \to \infty$.
  Moreover, $\alpha_\gamma \to \alpha \in (0, 1]$.
  \label{ass:bandwidth-1}
  \item (Superpopulation model).
  \((Y_{\gamma i},Z_{\gamma i})\) are i.i.d.\ across \(i \in \fp\) with $Y_{\gamma i}\sim g \in L^1(\mathbb R^d)$.
  The design may depend on \(Z\) but not directly on \(Y\) given \(Z\) (PPS).
  \label{ass:superpopulation}
  \item (Design regularity).
  There exists $0 < c_0<\infty$ such that
  $$
    \lim_{\gamma\to\infty}\,
    \Pr\!\left(
    \max_{i \in \fp}\pi_{\gamma i}^{-1}\ \le\ c_0/\alpha_\gamma
    \right)=1.
  $$
  Equivalently, we write $\ \max_i \pi_{\gamma i}^{-1}=O_p(1/\alpha_\gamma)$.
  To satisfy this assumption in applications, extremely large inverse inclusion weights can be truncated.
  \label{ass:design-regularity}
\end{enumerate}

Lemma~\ref{lem:self-normalization} in the Appendix shows that under these assumptions $\hat f_\gamma(y)$ is self-normalizing, i.e., integrates to 1 for every sample.
The following theorem shows that the HT-adjusted KDE converges to the true density $g$ in $L^1$.

\begin{theorem}[Large-deviation-based \(L_1\)-consistency of HT-adjusted KDE]\label{thm:ht-kde-consistency}
Under Assumptions \ref{ass:kernel-1}--\ref{ass:design-regularity},
\[
\|\hat f_{\gamma}-g\|_{L^1}\ \xrightarrow{\;\Pr\;}\ 0.
\]
Moreover, there exist constants \(C_1,C_2,C_3>0\), depending only on \(K\) and \(c_0\),
such that for all \(\tau\in(0,1]\),
$$
\Pr\left(\|\hat f_{\gamma}-g\|_{1} > \tau \right)
\le
C_1 \exp\left\{-C_2 \neff h_\gamma^{d} \tau^2\right\}
+
C_1 \exp\left\{-C_3 N_\gamma h_\gamma^{d} \tau^2\right\}
+
o(1).
$$
If in addition $\neff h_\gamma^d / \log(1/h_\gamma)\to\infty$, then
\(\|\hat f_{\gamma}-g\|_{1}\to 0\) almost surely.
\end{theorem}

A key ingredient in the proof of the $L^1$ consistency is the following proposition about the large-deviation bounds for the design term.

\begin{proposition}[Direct large-deviation bounds for the design term]\label{prop:design-ld}
Under assumptions~\ref{ass:kernel-1} and \ref{ass:design-regularity}, and letting
\begin{align*}
\bar f_{\gamma,h}(y):=\frac{1}{N_\gamma}\sum_{i\in\fp}K_{h_\gamma}(y-Y_{\gamma i}),
\end{align*}
there exist constants $c,C>0$ (depending only on $c_0,K,d$) such that for all $\tau\in(0,1]$,
$$
\Pr\left(\|T_\gamma-\bar f_{\gamma,h}\|_{L^1}>\tau\ \mid \{Y,Z\}\right)
\le
C\exp\left\{-c\, \neff\, h_\gamma^{d}\,\min(\tau^2,\tau)\right\}
+
C\exp\{-c\, N_\gamma\, h_\gamma^{d}\}.
$$
Under sampling without replacement (rejective), the first exponent is multiplied by $(1-\alpha_\gamma)$.
\end{proposition}

The proof of the proposition as well as the consistency of the HT-adjusted KDE are given in Appendix~\ref{sec:appendix-proof-consistency}.

\begin{remark}[Rates under smoothness]
If $g$ is $\beta$-Hölder and $K$ has order $\beta$, the three-way decomposition in the large-deviation bound yields
$$
\|\hat f_{\gamma}-g\|_1
= O_\Pr\left(h_\gamma^\beta\right)
+ O_\Pr\left((\neff\,h_\gamma^d)^{-1/2}\right)
+ O_\Pr\left((N_\gamma h_\gamma^d)^{-1/2}\right).
$$
Since $N_\gamma\ge n_{\mathrm{eff},\gamma}$, the last term is dominated by the middle one. 
Choosing $h_\gamma\asymp n_{\mathrm{eff},\gamma}^{-1/(2\beta+d)}$ balances the (design) variance and bias, giving
$$
\|\hat f_{\gamma}-g\|_1=O_\Pr\left(\neff^{-\beta/(2\beta+d)}\right).
$$
\end{remark}

\begin{corollary}[Simple random sampling]
If \(\pi_{\gamma i}\equiv n_\gamma/N_\gamma\) (i.e., simple random sampling with replacement), then \(\neff \asymp n_\gamma\)
and Theorem~\ref{thm:ht-kde-consistency} recovers the well-known triangular-array SRS result:
if \(h_\gamma\downarrow 0\) and \(n_\gamma h_\gamma^d\to\infty\), then
\(\|\hat f_{\gamma}-g\|_1\to 0\) in probability.
\end{corollary}

\subsection{Consistency of the MHDE}

The MHDE with HT-adjusted KDE plug-in~\eqref{eqn:mhde-def} is equivalent to any measurable maximizer
$\hat\theta_\gamma \in \argmax_{\theta\in\Theta}\Gamma_\gamma(\theta)$,
and we make the following identifiability assumption:
\begin{enumerate}[label=\textbf{A\arabic*}, resume=assumptions]
  \item (Identifiability).
  $\theta_0$ uniquely maximizes $\Gamma(\theta)$ and, for each $\varepsilon>0$,
$\sup_{\|\theta-\theta_0\|\ge \varepsilon}\Gamma(\theta) \le \Gamma(\theta_0)-\Delta(\varepsilon)$ for some $\Delta(\varepsilon)>0$.
  \label{ass:mhde-identifiability}
\end{enumerate}

The following proposition establishes the tail bounds for the MHDE deviations and is proven in Appendix~\ref{sec:appendix-proof-consistency-mhde}.

\begin{proposition}[Exponential tail bounds for uniform MHDE deviation]\label{prop:mhde-tail}
Under Assumptions \ref{ass:kernel-1}--\ref{ass:design-regularity}, for all $t\in(0,1]$,
\begin{equation}\label{eq:mhde-tail}
\Pr\left(\sup_{\theta}|\Gamma_\gamma(\theta)-\Gamma(\theta)|>t\right)
\le
C \exp\left\{c\,\neff h_\gamma^{d}\,\min\{t^4,t^2\}\right\}
+
C \exp\left\{-c\,N_\gamma h_\gamma^{d}\right\}.
\end{equation}
For Poisson--PPS without replacement, multiply the first exponent by $(1 - \alpha_\gamma)$.
\end{proposition}

\begin{theorem}[Consistency of MHDE with HT plug-in]\label{thm:mhde-consistency}
Under Assumptions \ref{ass:kernel-1}--\ref{ass:mhde-identifiability}, let $\hat\theta_\gamma$ be any sequence of $\varepsilon_\gamma$-maximizers, i.e.,
$$
\Gamma_\gamma(\hat\theta_\gamma)\ge \sup_{\theta}\Gamma_\gamma(\theta) - \varepsilon_\gamma
$$
with $\varepsilon_\gamma\to 0$.
Then $\hat\theta_\gamma \to \theta_0$ in probability.
Moreover, for every $\varepsilon>0$,
$$
\Pr\left(\|\hat\theta_\gamma-\theta_0\|\ge \varepsilon\right)
\le
\Pr\left(\sup_{\theta}|\Gamma_\gamma(\theta)-\Gamma(\theta)|>\tfrac{1}{3}\Delta(\varepsilon)\right)
+
\1\{\varepsilon_\gamma>\tfrac{1}{3}\Delta(\varepsilon)\}.
$$
In particular, using \eqref{eq:mhde-tail}, the RHS decays at the exponential rate in $n_{\mathrm{eff},\gamma}h_\gamma^d$.
\end{theorem}

\begin{proof}
Fix $\varepsilon>0$ and write $\Delta=\Delta(\varepsilon)$ from Assumption~\ref{ass:mhde-identifiability}.
For the event
\[
\mathcal E_\gamma(\varepsilon)
:=
\left\{ \{\sup_{\theta}|\Gamma_\gamma(\theta)-\Gamma(\theta)|\le \tfrac{1}{3}\Delta\} \cap \{ \varepsilon_\gamma\le \tfrac{1}{3}\Delta\} \right\},
\]
we claim that every $\varepsilon_\gamma$-maximizer $\hat\theta_\gamma$ lies in $B(\theta_0,\varepsilon)$.
In fact, for any $\theta$ with $\|\theta-\theta_0\|\ge \varepsilon$,
\[
\Gamma_\gamma(\theta) \le \Gamma(\theta)+\tfrac{1}{3}\Delta \le \Gamma(\theta_0)-\tfrac{2}{3}\Delta,
\]
while at $\theta_0$,
\(
\Gamma_\gamma(\theta_0) \ge \Gamma(\theta_0)-\tfrac{1}{3}\Delta.
\)
Thus $\sup_{\|\theta-\theta_0\|\ge \varepsilon}\Gamma_\gamma(\theta)\le \Gamma_\gamma(\theta_0)-\tfrac{1}{3}\Delta$.
If $\hat\theta_\gamma$ is an $\varepsilon_\gamma$-maximizer with $\varepsilon_\gamma\le \Delta/3$, then
\[
\Gamma_\gamma(\hat\theta_\gamma)\ \ge\ \sup_\theta \Gamma_\gamma(\theta)-\varepsilon_\gamma
\ >\ \Gamma_\gamma(\theta_0) - \tfrac{1}{3}\Delta,
\]
so $\hat\theta_\gamma$ cannot lie outside $B(\theta_0,\varepsilon)$.
Therefore
\[
\Pr\left(\|\hat\theta_\gamma-\theta_0\|\ge \varepsilon\right)
\le
\Pr\left(\mathcal E_\gamma(\varepsilon)^c\right)
\le
\Pr\left(\sup_{\theta}|\Gamma_\gamma(\theta)-\Gamma(\theta)|> \tfrac{1}{3}\Delta\right)
+\1\{\varepsilon_\gamma>\Delta/3\}.
\]
Applying Proposition~\ref{prop:mhde-tail} yields the RHS $\to 0$ as $\gamma\to\infty$.
\end{proof}

\begin{remark}
The argument uses only uniform (in $\theta$) control via Lemma~\ref{lem:uniform} and separation in Assumption~\ref{ass:mhde-identifiability}.
Compactness of $\Theta$ or upper semicontinuity of $\Gamma$ are not required.
\end{remark}

\begin{remark}
All statements hold verbatim under Poisson--PPS without replacement.
In the tail bound \eqref{eq:mhde-tail}, multiply the first exponent by $(1-\alpha_\gamma)$ (finite-population correction).
\end{remark}

\begin{remark}
In case $g \in \mathscr{F}$ such that $g \equiv f_0$, $\hat\theta_\gamma$ converges to the true parameter $\theta_0$.
\end{remark}

\subsection{Asymptotic normality}

To derive the central limit theorem for $\hat\theta_\gamma$ we need the following additional assumptions.

\begin{enumerate}[label=\textbf{A\arabic*}, resume=assumptions]
\item (Design)\label{ass:clt-design}
For Poisson--PPS, $\sqrt{\tneff(S_\gamma - 1)} = O_\Pr(1)$ and
$$
\max_{i\in\fp} \frac{(1 - \pi_{\gamma i}) / \pi_{\gamma i}} { \sum_{j \in \fp} (1 - \pi_{\gamma i}) / \pi_{\gamma i} } \xrightarrow{\;\Pr\;} 0.
$$
For fixed-size SRS--WOR, with $\tneff \equiv n_\gamma / (1 - \alpha_\gamma)$ and $S_\gamma \equiv 1$, these are satisfied if the finite-population correction (FPC) factor $(1 - \alpha_\gamma) \to (1 - \alpha)$ with $\alpha \in [0, 1)$.

\item (Kernel approximation)\label{ass:clt-kernel}
Either $|1 - K(\xi)| \lesssim |\xi|^\beta$ as $\xi \to 0$ for some $\beta > 0$; or $K$ has order $\beta > 0$ (i.e., vanishing moments up to $\lfloor \beta \rfloor)$ and $\sqrt g \in \mathcal C^\beta$ on compacta.
In both cases, $K_h * \phi_g \to \phi_g$ in $L^2(g)$ and $g * K_h - g = O(h_\gamma^\beta)$ in the weighted sense used in the theorem below.

\item (Model smoothness and identifiability)\label{ass:clt-smooth}
$\theta_0$ is the unique maximizer of $\Gamma(\theta)$ with positive definite Hessian $\mat A$.
Moreover, $\phi\in L^2(g;\R^p)$ and $\theta\mapsto \sqrt{f_\theta(y)}$ is twice continuously differentiable in a neighborhood of $\theta_0$, with dominated derivatives allowing differentiation under the integral for $\Gamma$ and $\Gamma_\gamma$.

\item (Bandwidth regime)\label{ass:clt-bandwidth}
The bandwidth $h_\gamma \downarrow 0$ and
\begin{align*}
\sqrt{\tneff}\,h_\gamma^{2\beta}\to 0,
&&
\frac{1}{\sqrt{\tneff}\,h_\gamma^{d}}\to 0,
&&
\frac{\sqrt{\tneff}}{N_\gamma\,h_\gamma^{d}}\to 0.
\end{align*}

\item (Localized risk control)\label{ass:clt-local-risk}
Fix an exhaustion by compacta $A_R\uparrow \R^d$ with $c_R:=\inf_{A_R} g > 0$ (automatic if $g$ is continuous and strictly positive).
For each $R$, there exist $h_0(R)>0$, $C_R<\infty$, and a tail remainder $\tau_R(h)\downarrow 0$ (as $R\uparrow\infty$ uniformly on $h\le h_0(R)$) such that for all $0<h\le h_0(R)$,
\begin{align*}
\sup_{t\in\R^d}\ \int_{A_R}\frac{K_h^2(y-t)}{g(y)}\dx y\ \le\ C_R\,h^{-d},
&&
\sup_{0<h\le h_0(R)}\ \int_{A_R^c}\frac{(K_h^2*g)(y)}{g(y)}\dx y\ \le\ \tau_R(h).
\end{align*}
In addition, Assumption~\ref{ass:clt-kernel} implies the bias bound on compacta
\begin{align*}
\int_{A_R}\frac{(g*K_h-g)^2}{g}\dx y\ \lesssim_R\ h^{2\beta},\,\text{and}\;
\sup_{0<h\le h_0(R)}\int_{A_R^c}\frac{(g*K_h-g)^2}{g}\dx y\ \le\ \tau_R(h).
\end{align*}

\item (Lindeberg/no dominant unit)\label{ass:clt-lindeberg}
Let $\psi_{h_\gamma}=K_{h_\gamma}*\phi_{g}$ and define
$$
X_{\gamma i}:=\frac{1}{N_\gamma}\left(\frac{\delta_{\gamma i}}{\pi_{\gamma i}}-1\right)\psi_{h_\gamma}(Y_{\gamma i})\in\R^p.
$$
Assume the Lindeberg condition for triangular-arrays holds conditional on $\{Y_{\gamma i}\}$, i.e., for every $\varepsilon>0$,
$$
\frac{\sum_{i\in\fp}\E\left[\|X_{\gamma i}\|^2\,\1\left\{\|X_{\gamma i}\|>\varepsilon/\sqrt{\tneff} \right\} \MID \{Y_{\gamma i}\}\right]}
{\sum_{i\in\fp}\Var(X_{\gamma i}\mid \{Y_{\gamma i}\})}
\ \xrightarrow{\;\Pr\;}\ 0,
$$
and $\sum_i\Var(X_{\gamma i}\mid \{Y\})$ converges in probability to $\bm\Sigma/\tneff$ (see the variance limit below).
\end{enumerate}

Assumption~\ref{ass:clt-design} is standard and mild for Poisson--PPS and automatically satisfied for SRS--WOR.
The assumptions~\ref{ass:clt-kernel} and~\ref{ass:clt-bandwidth} are standard for KDE in finite populations, while~\ref{ass:clt-local-risk} is substantially weaker than the usual assumption of $\inf {g} > 0$ globally.
A sufficient condition for~\ref{ass:clt-lindeberg} to hold is $\E_{G}\|\phi(Y)\|^{2+\eta}<\infty$ for some $\eta>0$ together with the no-dominant-unit condition in Assumption~\ref{ass:clt-design}.

\begin{corollary}
Under assumptions~\ref{ass:clt-design}, \ref{ass:clt-kernel} and \ref{ass:clt-local-risk} with $\psi_{h_\gamma}=K_{h_\gamma}*\phi_{g}$ we have
$$
\tneff\Var\left(\frac{1}{N_\gamma}\sum_{i\in\fp}\left(\frac{\delta_{\gamma i}}{\pi_{\gamma i}}-1\right)\psi_{h_\gamma}(Y_{\gamma i}) \MID \{Y_{\gamma i}\}\right)\ \xrightarrow{\;\Pr\;}\ \bm\Sigma.
$$
For SRS--WOR, the variance is multiplied by the FPC $(1-\alpha_\gamma)$.
\end{corollary}

\begin{theorem}\label{thm:clt}
Under assumptions~\ref{ass:clt-design}--\ref{ass:clt-lindeberg} the asymptotic distribution of the MHDE $\hat\theta_
\gamma$ under Poisson--PPS is Gaussian:
$$
\sqrt{\tneff} (\hat\theta_\gamma - \theta_0) \Rightarrow 
  N_p(\mat 0, \mat A^{-1} \bm\Sigma \mat A^{-\intercal}).
$$
For fixed-size SRS--WOR, the covariance matrix must be multiplied by the FPC $(1 - \alpha_\gamma)$.
\end{theorem}

The proof borrows ideas from \citet{Cheng2006} and is given in Appendix~\ref{sec:proof-clt}.
Here we want to discuss a few important insights from the proof technique.

\begin{remark}
Our proof does not require a global lower bound on ${f_0}$.
All variance and bias controls in assumption~\ref{ass:clt-local-risk} localized on $A_R$ with a tail remainder $\tau_R(h)$ that can be driven to 0 by taking $R=R_\gamma \uparrow \infty$ slowly, uniformly over $h \leq h_0(R)$.
\end{remark}
\begin{remark}
The bandwidth assumption~\ref{ass:clt-bandwidth} is necessary to remove the kernel bias, the i.i.d.\ KDE smoothing noise of order $1 / \sqrt{N_\gamma h_\gamma^d}$ and also to leave the HT design fluctuation at the scale of $\sqrt{\tneff}$.
\end{remark}
\begin{remark}
Under SRS the assumptions reduce to the classical conditions $\sqrt{n_\gamma} h_\gamma^d \to \infty$ and $\sqrt{n_\gamma} h_\gamma^{2\beta}$, up to the FPC, as in i.i.d.\ MHDE analyses without global lower bound on $f$.
\end{remark}
\begin{remark}
For fixed-size PPS--WOR, assumption~\ref{ass:clt-design} would need to be replaced with the usual rejective design with $\sum_{i\in\fp}\delta_{\gamma i} = n_\gamma$ and first-order $\{\pi_{\gamma i}\}$.
The same FPC factor as with SRS--WOR appears asymptotically, and the rest of the statement is unchanged.
\end{remark}

\subsection{Robustness}

Finally, we turn to the robustness of the MHDE against contaminated superpopulations.
We define the estimator functional
$$
T(G) \in \argmax_{\theta \in \Theta} \Gamma_G(\theta),
$$
and the gradient of the population-level $\Gamma_G$ as $S_G(\theta) := \nabla_\theta \Gamma_G(\theta) = \frac{1}{2} \int u_\theta(y) \sqrt{f_\theta(y) g(y)}.
$
The MHDE~\eqref{eqn:mhde-def} targets $\theta_0 = T(G)$.
Note that the sampling design does not affect the functional, only the estimator.
Hence, the design does not affect the robustness properties.

We denote the contaminated superpopulation distribution by $G_\epsilon := (1 - \epsilon) G + \epsilon H$ with arbitrary contamination distribution $H$.
We work in the Hellinger topology, $H(f_1, f_2) := \| \sqrt{f_1} - \sqrt{f_2} \|_{L^2}$ and make the following assumptions about the model smoothness.

\begin{enumerate}[label=\textbf{A\arabic*}, resume=assumptions]
\item (Model smoothness and identifiability).\label{ass:infl-model}
\begin{enumerate}[label=(\roman*)]
    \item
    For each $y$, $\theta\mapsto \sqrt{f_\theta(y)}$ is twice continuously differentiable in a neighborhood $\mathscr N(\theta_0)$ of $\theta_0=T(G)$.

    \item
    There exists an envelope $e \in L^2(g)$ such that for all $\theta\in\mathscr N(\theta_0)$,
\begin{align*}
\left\|u_\theta \sqrt{f_\theta} \right\|_{L^2(g)} &\le \Big\| e \Big\|_{L^2(g)}, \\
\left\| \partial_\theta u_\theta \sqrt{f_\theta} \right\|_{L^2(g)} &\le \Big\| e \Big\|_{L^2(g)}.
\end{align*}

    \item
    $\theta_0$ is the unique maximizer of $\Gamma$ and there is a separation margin: $\forall \varepsilon>0, \exists \Delta(\varepsilon)>0$ s.t.,
$$
\sup_{\|\theta-\theta_0\|\ge\varepsilon} \Gamma(\theta) \le \Gamma(\theta_0)-\Delta(\varepsilon).
$$
    \item\label{ass:infl-model-hessian}
    The Hessian $-\nabla_\theta^2\Gamma(\theta)\MID_{\theta=\theta_0}$ exists and is nonsingular.
\end{enumerate}

\item (Directional Gateaux derivative in $F_0$).\label{ass:infl-gateaux}
Let $H$ be a finite signed measure on $(\R^d,\mathcal B)$ with $\int \|u_{\theta_0}(y)\|\sqrt{f_{\theta}(y)}\,\frac{\mathrm d|H|(y)}{\sqrt{g(y)}}<\infty$ (e.g., $H \ll G$ with density in $L^2(g)$, or $H=\Delta_z$ with $g(z)>0$ and finite integrand).
The Gateaux derivative of $S_G$,
$$
\dot S_{G}(\theta; H)\ :=\ \lim_{\varepsilon\downarrow 0}\frac{S_{G_\varepsilon}(\theta)-S_{G}(\theta)}{\varepsilon}
$$
exists for $\theta$ near $\theta_0$ and
$$
\dot S_{G}(\theta;H) = \frac14 \int u_\theta(y)\,\sqrt{f_\theta(y)}\,\frac{dH(y)}{\sqrt{g(y)}}.
$$

\end{enumerate}

Based on these assumptions, we derive the influence function \citep{Hampel1974} and the $\alpha$-influence curve to describe the estimator's behavior under small levels of contamination.
The proofs of the following theorem and corollary are given in the Appendix~\ref{sec:proof-robust}.

\begin{theorem}[Influence function]\label{thm:influence}
Under assumptions~\ref{ass:infl-model}--\ref{ass:infl-gateaux}, the functional $T$ is Gateaux differentiable at $G$ in direction $H$ and has influence function
$$
\text{IF}(H; T, G) := \frac{\mathrm{d}}{\mathrm d \varepsilon} T(G_\varepsilon) \MID_{\varepsilon=0} = -\mat Q^{-1} \dot S_{G}(\theta_0; H),
$$
with 
\begin{align*}
\mat Q := \nabla_\theta S_{G}(\theta) \MID_{\theta=\theta_0}
= \frac12 \int \left[ \left[\nabla_\theta u_\theta(y)\right] +\frac{1}{2}\,u_\theta(y) u_\theta(y)\tr \right]_{\theta=\theta_0} \sqrt{f_{\theta_0}(y)\, g(y)} \dx y.
\end{align*}
In particular, for a point mass $H = \Delta_z$ with $g(z) > 0$,
$
\text{IF}(z; T, G) = -\mat Q^{-1} \phi_{g}(z).
$
\end{theorem}

\begin{corollary}[$\alpha$-influence curve]\label{cor:alpha-influence}
For $G_\varepsilon$ with small $\varepsilon$,
$$
T(G_\varepsilon) = \theta_0 - \varepsilon \mat Q^{-1} \dot S_G(\theta_0; H) + O(\varepsilon^2).
$$
In particular, for point-contamination at $z$, $H=\Delta_z$, with $g(z) > 0$, $T(G_\varepsilon) = \theta_0 - \varepsilon \mat Q^{-1} \phi_g(z) + O(\varepsilon^2)$.
\end{corollary}

\begin{remark}
All statements are made in the Hellinger topology.
The influence function holds for any direction $H$ satisfying Assumption~\ref{ass:infl-gateaux}.
In particular, for point-mass contamination $\Delta_z$, $g(z) > 0$ to avoid division by zero in $\sqrt{f_{\theta_0}(z) / g(z)}$.
For directions $H \ll G$ with density $h = \dx H/\dx G \in L^2(g)$, Assumption~\ref{ass:infl-gateaux} is always satisfied since $\int \| u_{\theta_0}\| \sqrt{f_{\theta_0}} h / \sqrt{g} \dx y = \int \| u_{\theta_0}\| \sqrt{f_{\theta_0} / g}\; h \dx G$ is finite under the $L^2(g)$ envelope.
\end{remark}

\section{Empirical Studies}\label{sec:empirical-studies}

To bring the theoretical properties derived above into perspective and compare with the maximum likelihood estimator, we conduct a large simulation study.
We then demonstrate the versatility of the MHDE~\eqref{eqn:mhde-def} by applying it in the National Health and Nutrition Examination Survey (NHANES)\citep{nhanes2025}.

\subsection{Simulation study}

We simulate data from a finite population of size $N\in\{10^6, 10^{6.5}, 10^7, 10^{7.5}, 10^8\}$ with two different sampling ratios $\alpha \in \{10^{-3}, 10^{-4}\}$.
The characteristic of interest follows a $\Gamma$ superpopulation model, $Y \sim \Gamma(2, 35\,000)$.
For Poisson--PPS we simulate a log-normal auxiliary variable $Z$ using different correlations with $Y$, $\rho_{YZ} \in \{0.25, 0.75\}$.
In Section~\ref{sec:suppl-gamma-calibrated} of the supplementary materials, we present the results with the survey weights calibrated to match known cluster totals.
The conclusions from the calibrated survey weights are similar to what is presented here.

To inspect the robustness properties of the MHDE, we introduce point-mass contamination in a fraction of the sampled observations.
Specifically, we replace $\lfloor \varepsilon n \rfloor$ observations in the sample with draws from a Normal distribution with mean $z \gg \E[Y]$ and variance $10^{-2} \Var(Y)$.
For ``independent contamination,'' the contaminated observations are chosen completely at random, while for ``high-leverage contamination,'' observations with higher sample weight are more likely to be contaminated, $\Pr(\text{obs. i is contaminated}) \propto (1 - \pi_{\gamma i})^{-10}$.
The supplementary materials (Section~\ref{sec:suppl-simres-tcont}) contain results for a scenario where the contamination comes from a truncated \textit{t} distribution with 3 degrees of freedom.

For each combination of simulation parameters, we present the relative absolute bias and the relative root mean square error across $R=100$ replications:
\begin{align*}
    \text{RelBias} := \frac{1}{\theta_0}\frac{1}{R} \sum_{r=1}^R(\hat\theta_r - \theta_0), &&
    \text{RelRMSE} := \frac{1}{\theta_0}\sqrt{ \frac{1}{R} \sum_{r=1}^R(\hat\theta_r - \theta_0)^2}.
\end{align*}

We compare the MHDE with the weighted maximum likelihood estimate (MLE) for the Gamma model.

\subsubsection{Results}

Figure~\ref{fig:sim-gamma-rel-bias} shows the relative bias of the MHDE and the MLE in the Gamma superpopulation model as the finite population size and the sample size increase.
When $N$ is sufficiently large, the bias is within $\pm 1\%$ for each parameter with both MHDE and MLE.
As expected, the variance of both estimators also decreases rapidly with increasing finite population size (Figure~\ref{fig:sim-gamma-rel-rmse}).

In the presence of contamination, the MHDE clearly shows its advantages over the MLE (Figure~\ref{fig:sim-gamma-infl}).
Overall, the estimates for the scale parameter of the Gamma superpopulation model are much more affected by contamination than the shape parameter.
Importantly, the influence function for the MHDE under independent and high-leverage contamination is bounded, whereas it is unbounded for the MLE.
From the $\alpha$-influence curve, we can further see that the MHDE can withstand up to 30\% high-leverage contamination before becoming unstable.
In the presence of independent contamination, on the other hand, the bias of the MHDE remains bounded even when approaching 50\% contamination.

We also verify the coverage of the asymptotic confidence intervals derived from Theorem~\ref{thm:clt} using $10\,000$ replications for $N=10^7$ and two sample sizes, $n \in \{1\,000, 10\,000\}$.
Table~\ref{tbl:sim-gamma-cis} summarizes the coverage and width of the 95\% confidence intervals for the different sampling strategies.
The coverage rate for SRS with and without replacement is very close to the nominal level.
For Poisson--PPS, on the other hand, the CI coverage is below the nominal level, likely due to the slightly higher bias observed also in Figure~\ref{fig:sim-gamma-rel-bias}.
However, this is not unique to the MHDE, but the MLE also suffers from the same issue in this setting.

In Section~\ref{sec:suppl-simres-lognormal} of the supplementary materials, we present a second simulation study using the log-normal distribution for $Y$.
The conclusions align with the Gamma model presented here, but the CI coverage is close to the nominal level for all sampling schemes.

\begin{figure}[t]
    \centering
    \includegraphics[width=1\linewidth]{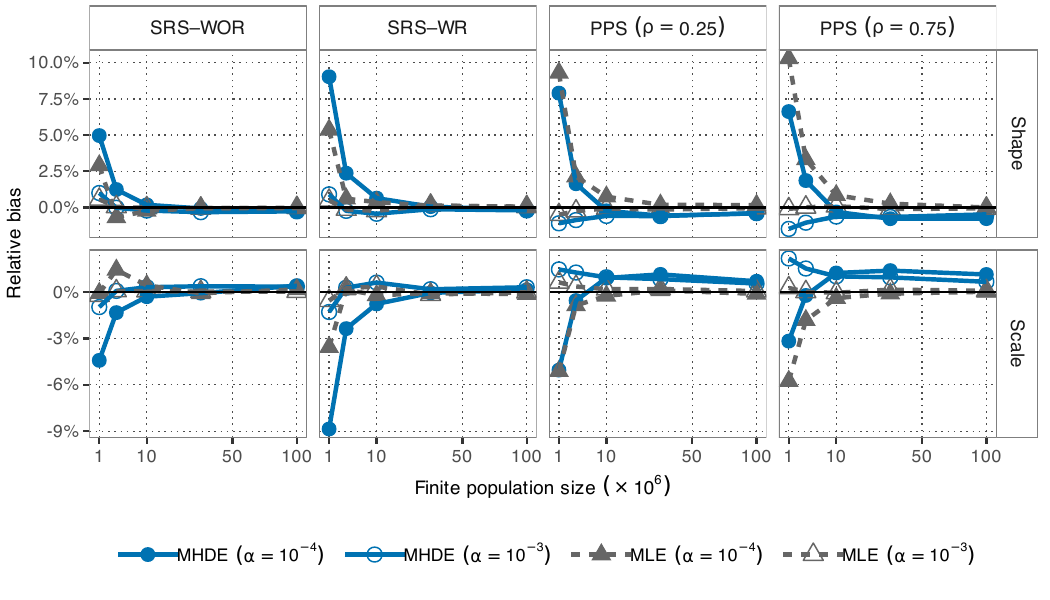}
    \caption{%
      Relative bias of the MHDE (blue dots) and MLE (gray triangles) in a Gamma superpopulation model using various sample designs.
      The sample size in each simulation is determined by $n=\alpha N$, with $\alpha \in \{10^{-3},10^{-4}\}$.%
      }
    \label{fig:sim-gamma-rel-bias}
\end{figure}

\begin{figure}[t]
    \centering
    \includegraphics[width=1\linewidth]{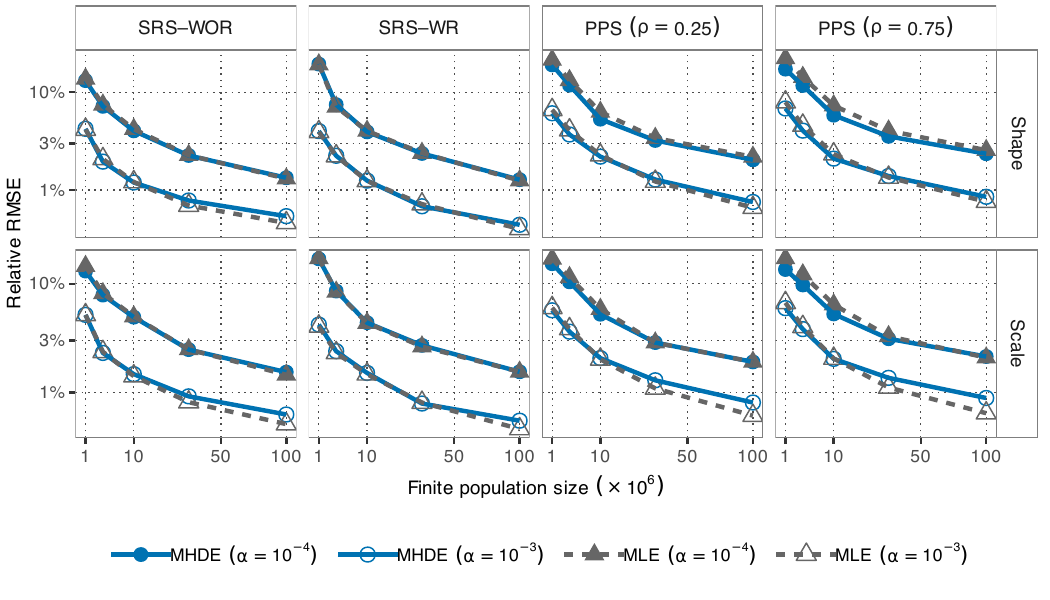}
    \caption{%
      Relative RMSE of the MHDE (blue dots) and MLE (gray triangles) under a Gamma superpopulation model using various sample designs.
      The sample size in each simulation is determine by $n=\alpha N$, with $\alpha \in \{10^{-3},10^{-4}\}$.%
      }
    \label{fig:sim-gamma-rel-rmse}
\end{figure}

\begin{figure}[t]
    \centering
    \includegraphics[width=1\linewidth]{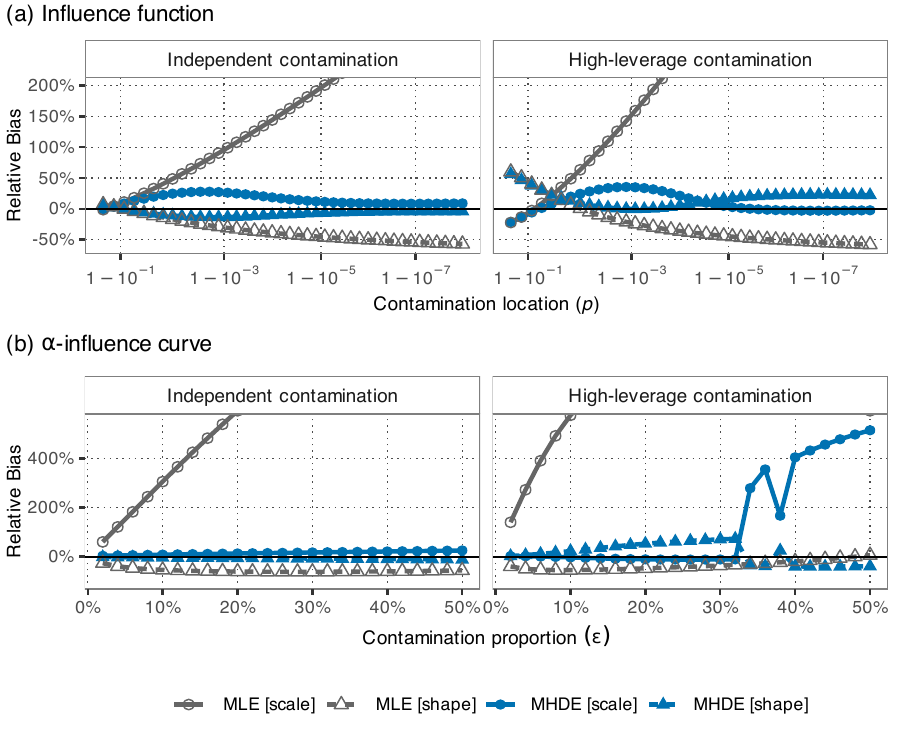}
    \caption{%
      Influence functions (top) and alpha curves (bottom) for the MHDE and MLE of the scale ($\circ$) and shape ($\triangle$) parameters in the Gamma model. %
      The contamination proportion in the influence function at the top is set to $\varepsilon=0.1$. 
      The horizontal axis shows the location of the point-mass contamination in terms of the quantile of the true superpopulation model, i.e., $z=G^{-1}(p)$. %
      For the alpha curve the point-mass contamination is located at $z=G^{-1}(1-10^{-7}) \approx 669\,193$.%
      }
    \label{fig:sim-gamma-infl}
\end{figure}

\begin{table}[t]
    \centering
    \begin{tabular}{rlrrrr}
    \toprule
                &                  & \multicolumn{2}{c}{Shape}     & \multicolumn{2}{c}{Scale}  \\
                                     \cmidrule(lr){3-4}              \cmidrule(lr){5-6}
    Sample size & Sampling scheme  & CI coverage & CI avg.\ width  & CI coverage & CI avg.\ width  \\
    \midrule
     $1\,000$ & PPS ($\rho=0.75)$ & 91.1\% & 9.9\% & 97.5\% & 11.5\% \\
              & SRS--WOR          & 95.8\% & 8.2\% & 94.3\% &  9.2\% \\
              & SRS--WR           & 95.4\% & 8.2\% & 94.0\% &  9.2\% \\
    \midrule
    $10\,000$ & PPS ($\rho=0.75)$ & 86.1\% & 3.1\% & 92.6\% &  3.6\% \\
              & SRS--WOR          & 95.2\% & 2.6\% & 94.9\% &  2.9\% \\
              & SRS--WR           & 95.2\% & 2.6\% & 94.8\% &  2.9\% \\
    \bottomrule
    \end{tabular}
    \caption{Coverage and average (relative) width of 95\% confidence intervals across $10\,000$ replicates of the MHDE estimates in the Gamma model with finite population size $N=10^7$.}
    \label{tbl:sim-gamma-cis}
\end{table}

\subsection{Application to NHANES}

We now analyze the total daily water consumption by U.S.\ residents as collected through the National Health and Nutrition Examination Survey (NHANES) \citep{nhanes2025}.
Over each 2-year period, NHANES surveys health, dietary, sociodemographic and other information from about 10,000 adults and children in the U.S.\ using several interviews, health assessments and other survey instruments spread over several days.
NHANES uses a complex survey design, and calibrated survey weights are reported separately for each part of the survey.
Here, we analyze the dietary interview data from the 2021--2023 survey cycle, specifically the total daily water consumption.
Each NHANES participant is eligible for two 24-hour dietary recall interviews.
In the 2021--2023 survey cycle, both interviews were conducted by telephone, as opposed to the first interview being conducted in person as in earlier iterations of NHANES.
This may decrease the reliability of the first interview compared to previous years.
In fact, three and four respondents reported drinking more than 10 liters a day in the first interview and the second interview, respectively.
These values are not only unusual, but can even lead to hyponatremia \citet{Adrogue2000}.

We fit a Gamma model and a Weibull model to the survey data, estimating the parameters using the proposed MHDE and the reference MLE.
Figure~\ref{fig:nhanes-day2} shows the fitted densities for these two models for the second day.
In both models, the MLE is shifted rightwards, apparently affected by the few unusually high values.
There is a single response of 44.2 liters/day with a sampling weight in the 99th percentile, which can have a devastating effect on the MLE.

In sample surveys, the interest is often in population statistics, like population averages or totals.
We can easily obtain these statistics and associated confidence intervals from the fitted superpopulation models.
Here, we estimate the effective sample size according to \citet{Kish1992} by $\widehat{n_\text{eff}} = (\sum_i w_i)^2 / (\sum_i w_i^2)$ since the inclusion probabilities are unknown.
In Table~\ref{tbl:nhanes-avg-estimates} we can again see that the MLE is shifted upward, likely due to the bias from the unreasonable outliers in the data.
The non-parametric estimates are computed using the weighted mean and median, with confidence intervals derived from the Taylor series expansion implemented in the survey R package \citet{Lumley2003}.
In general, the non-parametric estimates seem to agree with the MHDE estimates, with overlapping confidence intervals.
The ML estimates, on the other hand, are substantially higher.

\begin{figure}[t]
    \centering
    \includegraphics[width=1\linewidth]{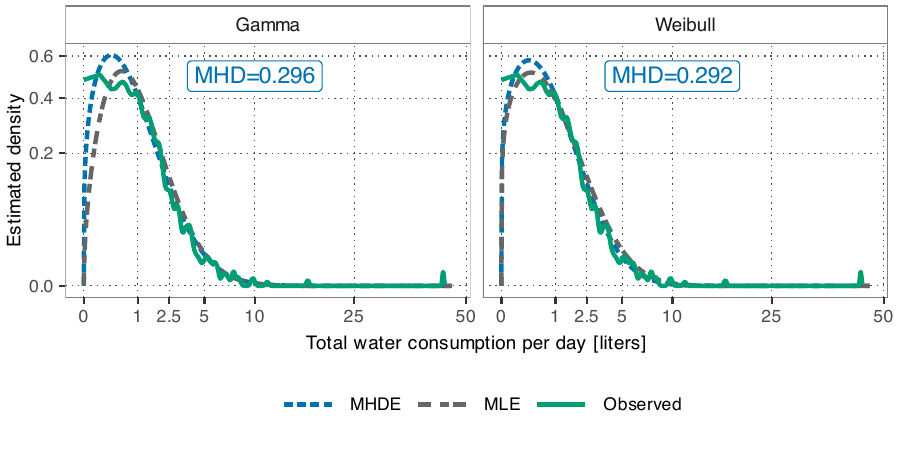}
    \caption{%
      MHDE and MLE estimates for two different parametric models to describe the total daily water consumption in the NHANES survey.
      The minimum Hellinger distance (MHD) is achieved by the MHDE shown here.%
      }
    \label{fig:nhanes-day2}
\end{figure}

\begin{table}[t]
    \centering
    \begin{tabular}{llrr}
    \toprule
    Model & Estimator & Mean [95\% CI] & Median [95\% CI] \\
    \midrule
    \multicolumn{2}{c}{\textit{Non-parametric}} 
                   & 1.26 [1.21, 1.31] & 1.01 [1.00, 1.08] \\
    Gamma   & MHDE & 1.32 [1.28, 1.36] & 0.99 [0.96, 1.02] \\
            & MLE  & 1.45 [1.41, 1.48] & 1.16 [1.13, 1.19] \\
    Weibull & MHDE & 1.32 [1.19, 1.35] & 1.02 [0.97, 1.17] \\
            & MLE  & 1.45 [1.37, 1.44] & 1.17 [1.20, 1.27] \\
    \bottomrule
    \end{tabular}
    \caption{%
        Population-level estimates of average total daily water consumption (in liters) using a non-parametric estimator and the MHDE and MLE for two different parametric models.
        The 95\% confidence intervals for the parametric estimates are Monte-Carlo approximations using $10\,000$ draws from the asymptotic distribution.%
    }
    \label{tbl:nhanes-avg-estimates}
\end{table}

\section{Discussion}

In this paper, we develop the Minimum Hellinger Distance Estimator (MHDE) with Horvitz-Thompson adjusted kernel density estimator for finite populations under various sampling designs.
In the superpopulation framework with potential model misspecification, we prove that the MHDE is consistent in the Hellinger topology and admits an asymptotic Normal distribution, with fully efficient covariance if the true distribution is in the parametric family.
We further derive the influence function and the $\alpha$-influence curve, showing that the MHDE is highly robust against contamination, including high-leverage points.
Our theory requires minimal assumptions on the true density and the sampling design, allowing for efficient estimation and valid inference even under post-stratification or calibration.
Hence, the MHDE is as efficient as the MLE if the superpopulation assumption is correct, but much more reliable and stable if the model is misspecified or the sample contaminated.

The MHDE is easy to implement for a wide class of parametric families with minimal adjustments.
In the numerical experiments, we applied the MHDE for the Gamma and Weibull models, but other models, such as the log-normal, are equally straightforward to implement.
The numerical experiments further underscore the utility of the MHDE in complex survey samples, particularly its stability under contamination and its versatility. 

The simplicity and efficiency our HT-adjusted MHDE make it an ideal candidate for complex survey samples and a wide range of superpopulation models.
While the focus in this paper is on the Hellinger distance, the techniques used in the proofs are expressively more general.
With appropriate adjustments to the assumptions, our results can be generalized to broader classes of divergences, such as power divergences \citep{Cressie1984} or $\phi$ divergences \citet{Pardo2006}.
A better understanding of the theoretical properties of more general HT-adjusted minimum divergence estimators is crucial to choosing the best estimator under different sampling strategies and contamination expectations.

\begin{appendices}

\section{Proof of Consistency}\label{sec:appendix-proof-consistency}

\subsection{Technical Lemmas}\label{sec:appendix-lemmas-consistency-kde}

\begin{lemma}[Self-normalization]\label{lem:self-normalization}
Let
\[
\widehat f_{\gamma}(y)=
\frac{\sum_{i\in\fp}\frac{\delta_{\gamma i}}{\pi_{\gamma i}}\,K_{h_\gamma}(y-Y_{\gamma i})}{\sum_{i\in\fp}\frac{\delta_{\gamma i}}{\pi_{\gamma i}}}
=\frac{T_\gamma(y)}{S_\gamma}.
\]

If $S_\gamma>0$, then $\int_{\mathbb R^d}\widehat f_{HT,\gamma}(y)\,dy=1$ almost surely.
\end{lemma}

\begin{proof}
By Fubini and the change of variables $u=(y-Y_{\gamma i})/h_\gamma$,
\begin{align*}
\int T_\gamma(y)\,dy
  &=\frac1{N_\gamma}\sum_{i=1}^{N_\gamma}\frac{\delta_{\gamma i}}{\pi_{\gamma i}}\int K_{h_\gamma}(y-Y_{\gamma i})\,dy \\
  &= \frac1{N_\gamma}\sum_{i=1}^{N_\gamma}\frac{\delta_{\gamma i}}{\pi_{\gamma i}}\int K(u)\,du
  =\frac1{N_\gamma}\sum_{i=1}^{N_\gamma}\frac{\delta_{\gamma i}}{\pi_{\gamma i}}=S_\gamma,
\end{align*}
since $K$ integrates to one by assumption~\ref{ass:kernel-1}.
Hence $\int \widehat f_{\gamma}=S_\gamma/S_\gamma=1$ on $\{S_\gamma>0\}$. Under Poisson–PPS with $n_\gamma=\sum_i\pi_{\gamma i}\to\infty$, $\Pr(S_\gamma=0)\le e^{-n_\gamma}\to0$; under sampling without replacement, $S_\gamma>0$ deterministically.
\end{proof}

\begin{lemma}[Bernstein, independent case]\label{lem:bernstein-ind}
Let $X_1,\dots,X_m$ be independent, $\E X_i=0$, $|X_i|\le b$, and 
$\sum_{i=1}^m \Var(X_i)\le v$. Then for all $t>0$,
\[
\Pr\Big(\sum_{i=1}^m X_i \ge t\Big)
\;\le\; \exp\!\left(-\,\frac{t^2}{2\,(v + b t/3)}\right),
\quad
\Pr\Big(\big|\sum_{i=1}^m X_i\big| \ge t\Big)\;\le\;2\,\exp\!\left(-\,\frac{t^2}{2\,(v + b t/3)}\right).
\]
\end{lemma}
The next lemma is a variant as applied to simple random sampling without replacement (SRSWOR).

\begin{lemma}[Bernstein under WOR (Serfling–type)]\label{lem:bernstein-wor}
Let a finite population $\{y_1,\dots,y_N\}\subset\mathbb R$ have mean $\mu=\frac1N\sum_{i=1}^N y_i$, range $|y_i-\mu|\le B$, and population variance
\(
\sigma_N^2=\frac1N\sum_{i=1}^N (y_i-\mu)^2.
\)
Draw a sample of size $m$ \emph{without replacement} and let $X_1,\dots,X_m$ be the sampled values in any order.
Then, with $v_{\text{wor}} := \frac{N-m}{N-1}\,m\,\sigma_N^2$, for all $t>0$,
$$
\Pr\left(\sum_{k=1}^m (X_k-\mu)\ \ge\ t\right) <
\exp\left(
-\frac{t^2}{2\left(v_{\text{wor}} + B t/3\right)}
\right).
$$
Equivalently, compared to the independent (with‑replacement) Bernstein bound with variance proxy $v_{\text{ind}}:=m\,\sigma_N^2$, the WOR bound holds with the \emph{finite‑population correction} $v_{\text{wor}}=\frac{N-m}{N-1}\,v_{\text{ind}}$.
\end{lemma}

\begin{proof}
Write $Z_k:=X_k-\mu$. Let $\mathcal F_k$ be the $\sigma$-field generated by the first $k$ draws, and consider the Doob (Hájek) decomposition
\[
M_k := \sum_{r=1}^k \left(Z_r - \E[Z_r\mid \mathcal F_{r-1}]\right),
\quad k=0,1,\dots,m,
\]
with $M_0=0$. Then $(M_k)_{k\le m}$ is a martingale and $M_m=\sum_{k=1}^m Z_k$ because $\sum_{k=1}^m \E[Z_k\mid \mathcal F_{k-1}]=0$ (the remaining population is always centered around its current mean and those means telescope to $0$; see Remark~\ref{rem:bernstein-wor-centering} below).

\paragraph*{Bounded increments.}
Each increment is bounded as
\[
|M_k-M_{k-1}| = \left| Z_k-\E[Z_k\mid \mathcal F_{k-1}] \right|
\le |Z_k|+|\E[Z_k\mid \mathcal F_{k-1}]|
\le 2B \ =: b .
\]

\paragraph*{Predictable quadratic variation.}
Define the predictable quadratic variation
\[
V_m := \sum_{k=1}^m \E\left[(M_k-M_{k-1})^2 \Mid \mathcal F_{k-1}\right]
=\sum_{k=1}^m \Var(Z_k\mid \mathcal F_{k-1}).
\]
Taking expectations and using the variance decomposition for martingales gives
\[
\E[V_m]=\Var\left(\sum_{k=1}^m Z_k\right).
\]
Under simple random sampling without replacement, the variance of the sample \emph{sum} is the classical finite‑population formula
$$
\Var\left(\sum_{k=1}^m X_k\right)
=
\frac{N-m}{N-1} m \sigma_N^2,
$$
hence $\E[V_m]=v_{\text{wor}}$ as claimed.

\paragraph*{Freedman's inequality and optimization.}
Freedman's inequality for martingales with bounded increments (e.g., Theorem 1.6 in \citet{Freedman1975}) states that for all $t,v>0$,
\[
\Pr\left(M_m\ge t \ \cap\ V_m\le v\right)\
\le
\exp\left\{-\frac{t^2}{2\,(v + b t/3)}\right\}.
\]
We use the standard peeling argument on the random $V_m$:
\begin{align*}
\Pr(M_m\ge t)
&= \sum_{j\ge 0} \Pr\left(M_m\ge t,\ V_m\in (2^{j-1}v_{\text{wor}}, 2^j v_{\text{wor}}]\right)\\
&\le\ \sum_{j\ge 0}\exp\left\{-\frac{t^2}{2\,(2^j v_{\text{wor}} + b t/3)}\right\}\\
&\le \ \exp\left\{-\frac{t^2}{2\,(v_{\text{wor}} + b t/3)}\right\},
\end{align*}
where the last inequality uses that the series is dominated by its first term (geometric decay in $j$ once $t$ is fixed). Substituting $b=2B$ and $v_{\text{wor}}$ yields the stated bound. The two‑sided tail follows by symmetry.

\end{proof}

\begin{remark}[Centering under WOR]\label{rem:bernstein-wor-centering}
At step $k$, conditional on $\mathcal F_{k-1}$, $X_k$ is uniformly distributed over the remaining $N-k+1$ units; its conditional mean equals the mean of the remaining values, which is $-\frac{1}{N-k+1}\sum_{r=1}^{k-1} Z_r$. Consequently, $\sum_{k=1}^m \E[Z_k\mid \mathcal F_{k-1}]=0$.
\end{remark}

\begin{remark}[Asymptotics and the $(1-\alpha)$ factor]\label{rem:bernstein-alpha-factor}
Since $\frac{N-m}{N-1}=(1-\alpha)\big(1-\tfrac{1}{N}\big)^{-1}$ with $\alpha=m/N$, the variance proxy satisfies $v_{\text{wor}}=(1-\alpha+o(1))\,m\sigma_N^2$ as $N\to\infty$.
Thus, in triangular‑array asymptotics with $N\to\infty$, replacing $v$ by $(1-\alpha)\,v$ in the independent Bernstein bound is correct up to a vanishing factor.
\end{remark}

\begin{lemma}[HT normalizer concentration]\label{lem:ht-normalizer-concentration}
Let $S_\gamma=\frac{1}{N_\gamma}\sum_{i=1}^{N_\gamma}\frac{\delta_{\gamma i}}{\pi_{\gamma i}}$. 
Under Poisson--PPS,
$$
\Pr\left(|S_\gamma-1|>t\right) \le 
2\exp\left(
-\frac{N_\gamma^2 t^2}{2\left(\sum_{i \in \fp}\frac{1-\pi_{\gamma i}}{\pi_{\gamma i}} + \frac{t N_\gamma}{3}\max_i \frac{1}{\pi_{\gamma i}}\right)}
\right)
\le
2\exp\left(-\,\frac{c \neff t^2}{1+c't}\right),
$$
for constants $c,c'>0$ under the regularity $\max_i \pi_{\gamma i}^{-1}\lesssim 1/\alpha_\gamma$. 
For WOR, multiply the denominator's variance term by $(1-\alpha_\gamma)$.
\end{lemma}

\begin{proof}
Write $S_\gamma-1=N_\gamma^{-1}\sum_i X_i$ with $X_i=(\delta_{\gamma i}/\pi_{\gamma i}-1)$, so $\E X_i=0$, $|X_i|\le \max_i \pi_{\gamma i}^{-1}$, 
$\sum \Var(X_i)=\sum (1-\pi_{\gamma i})/\pi_{\gamma i}$.
Apply Lemma~\ref{lem:bernstein-ind} with $t\gets N_\gamma t$; for WOR use Lemma~\ref{lem:bernstein-wor}.
\end{proof}

The next lemma is useful in establishing the first step of the proof of the Theorem~\ref{thm:ht-kde-consistency}.

\begin{lemma}[Three-way $L_1$ decomposition]\label{lem:threeWay}
Let
\begin{align*}
T_\gamma(y)=\frac{1}{N_\gamma}\sum_{i=1}^{N_\gamma}\frac{\delta_{\gamma i}}{\pi_{\gamma i}}\,K_{h_\gamma}(y-Y_{\gamma i}),
&&
S_\gamma=\frac{1}{N_\gamma}\sum_{i=1}^{N_\gamma}\frac{\delta_{\gamma i}}{\pi_{\gamma i}},
&&
\hat f_{\gamma}=\frac{T_\gamma}{S_\gamma},
\end{align*}
and let $\bar f_{\gamma,h} (y)=\frac{1}{N_\gamma}\sum_{i \in \fp}K_{h_\gamma}(y-Y_{\gamma i})$ be the i.i.d.\ kernel average.
Then
\begin{equation}\label{eq:threeWayDisplay}
\|\hat f_{\gamma}-f\|_{L^1}
\le
|S_\gamma^{-1}-1|
+
\|T_\gamma-\bar f_{\gamma,h}\|_{L^1}
+
\|\bar f_{\gamma,h}-f\|_{L^1}.
\end{equation}
\end{lemma}

\begin{proof}
Add and subtract $\bar f_{\gamma,h}$ and use the triangle inequality:
$$
\|\hat f_{\gamma}-f\|_1
\le
\|\hat f_{\gamma}-\bar f_{\gamma,h}\|_1
+
\|\bar f_{\gamma,h}-f\|_1.
$$
For the first term,
\[
\hat f_{\gamma}-\bar f_{\gamma,h}
=\frac{T_\gamma}{S_\gamma}-\bar f_{\gamma,h}
=\left(\frac{1}{S_\gamma}-1\right)\bar f_{\gamma,h}
+\frac{1}{S_\gamma}\left(T_\gamma-\bar f_{\gamma,h}\right),
\]
so by subadditivity of $\|\cdot\|_1$,
\[
\|\hat f_{\gamma}-\bar f_{\gamma,h}\|_1
\le
\left|\tfrac{1}{S_\gamma}-1\right|\ \|\bar f_{\gamma,h}\|_1
+
\tfrac{1}{S_\gamma}\ \|T_\gamma-\bar f_{\gamma,h}\|_1.
\]
Now $\|\bar f_{\gamma,h}\|_1=\int \bar f_{\gamma,h}(y) \dx y
= \frac1{N_\gamma}\sum_i \int K_{h_\gamma}(y-Y_{\gamma i})\dx y
=1$ because $\int K=1$.
Also, $S_\gamma>0$ with probability $1-o(1)$ (Poisson--PPS) or deterministically (WOR), and on $\{S_\gamma>0\}$ we may write
\[
\tfrac{1}{S_\gamma}\le 1+\left|\tfrac{1}{S_\gamma}-1\right|.
\]
Hence,
\[
\|\hat f_{\gamma}-\bar f_{\gamma,h}\|_1
\ \le\
\left|\tfrac{1}{S_\gamma}-1\right|
+\left(1+\left|\tfrac{1}{S_\gamma}-1\right|\right)\,\|T_\gamma-\bar f_{\gamma,h}\|_1.
\]
Finally, since $|1/S_\gamma-1|\,\|T_\gamma-\bar f_{\gamma,h}\|_1\ge 0$, we may drop that product term to obtain the simpler bound
\[
\|\hat f_{\gamma}-\bar f_{\gamma,h}\|_1
\ \le\
\left|\tfrac{1}{S_\gamma}-1\right|
+\|T_\gamma-\bar f_{\gamma,h}\|_1.
\]
Combine with the first display to conclude \eqref{eq:threeWayDisplay}.
\end{proof}

\subsection{Large-deviation bounds for the design term}\label{sec:appendix-design-ld}

To prove Proposition~\ref{prop:design-ld} we need the following technical lemmas.
Throughout this proof we use the definition of $\bar f_{\gamma,h}$ from Proposition~\ref{prop:design-ld}, set \(\xi_{\gamma i}:=\delta_{\gamma i}/\pi_{\gamma i}-1\) and define the signed measure
$$
\mu_\gamma:=\frac{1}{N_\gamma}\sum_{i\in\fp}\xi_{\gamma i} \delta_{Y_{\gamma i}},
$$
such that $T_\gamma-\bar f_{\gamma,h}=K_{h_\gamma}*\mu_\gamma$.
We further partition \(\R^d\) into half-open cubes \(\{B_{\gamma j}\}_{j=1}^{M_\gamma}\) with sides of length \(h_\gamma\) and centers \(c_{\gamma j}\).

\begin{lemma}[Cellwise smoothing reduction]\label{lem:cell-smooth}
Under assumption~\ref{ass:kernel-1} and defining
$$
S_\gamma(y):=(T_\gamma-\bar f_{\gamma,h})(y) = (K_{h_\gamma}*\mu_\gamma)(y),
$$
there exist constants $C_K,c,C'>0$ (depending only on $K$ and $d$) such that
\begin{align*}
\int_{\R^d} |S_\gamma(y)|\dx y
&\le
C_K\sum_{j=1}^{M_\gamma}\left|\mu_\gamma(B_{\gamma j})\right|
+ R_\gamma,
\;\text{and}\\
\Pr\left(R_\gamma>\tau\right) &\le C' e^{-c N_\gamma h_\gamma^{d}\,\min(\tau^2,\tau)}.
\end{align*}
\end{lemma}

\begin{proof}[Proof sketch.]
Decompose $\mu_\gamma$ into its restrictions on the cells and replace each atom in $B_{\gamma j}$ by a single atom at $c_{\gamma j}$; integrate the Lipschitz translation error of $K_{h_\gamma}$ over each cell. The sum of these errors concentrates with $N_\gamma h_\gamma^{d}$ by Bernstein applied to i.i.d.\ occupancies of the cells (details as in Lemma~\ref{lem:bernstein-ind}).
\end{proof}

\begin{lemma}[Convolution reduction on a grid]\label{lem:conv-grid}
Under assumptions~\ref{ass:kernel-1} and~\ref{ass:bandwidth-1}, for any finite signed measure $\mu$,
\begin{align*}
\|K_h * \mu\|_{L^1}
&\le
\|K\|_{L^1}\sum_{j} |\mu(B_j)|
+
C_K \sum_{j} |r_j|(B_j),\\
\text{with }
r_j &:= \mu\!\restriction_{B_j} - \mu(B_j)\,\delta_{c_j},
\end{align*}
where one may take \(C_K:=\|\nabla K\|_{L^1}\,\frac{\sqrt d}{2}\) (so \(C_K\) does \textbf{not} depend on \(h\)).
In particular, since \(|r_j|(B_j)\le 2\,|\mu|(B_j)\),
\[
\|K_h * \mu\|_{L^1}
\le
\|K\|_{L^1}\sum_{j} |\mu(B_j)|
+ 2 C_K \sum_{j} |\mu|(B_j).
\]
\end{lemma}

\begin{proof}
Write $\mu=\sum_j \nu_j$ with $\nu_j=\mu\!\restriction_{B_j}$ and decompose
$\nu_j=\mu(B_j)\,\delta_{c_j}+r_j$ with $r_j(B_j)=0$.
Then
$$
K_h * \mu = \sum_j \mu(B_j) K_h(\cdot - c_j) + \sum_j K_h * r_j.
$$
Taking $L^1$ norms gives the first term as $\|K\|_{L^1}\sum_j |\mu(B_j)|$.
For the remainder, by Minkowski and the translation inequality valid for $K\in W^{1,1}$,
\begin{align*}
\|K_h * r_j\|_{L^1}
&\le
\int_{B_j} \|K_h(\cdot - t) - K_h(\cdot - c_j)\|_{L^1}\, |\dx r_j|(t) \\
&\le
\left(\sup_{t\in B_j}\|K_h(\cdot - t) - K_h(\cdot - c_j)\|_{L^1}\right)\,|r_j|(B_j).
\end{align*}
Now
\(
\|K_h(\cdot - t) - K_h(\cdot - c_j)\|_{L^1}
\le \|\nabla K\|_{L^1}\|t-c_j\|/h
\le \|\nabla K\|_{L^1}(\sqrt d/2)
\)
because $\|t-c_j\|\le (\sqrt d/2)h$ for $t\in B_j$.
Summing over $j$ yields the claim.
\end{proof}

\begin{lemma}[Per-cell Bernstein/Serfling bound]\label{lem:cell-bernstein}
Let $A_{\gamma ij}:=\1\{Y_{\gamma i}\in B_{\gamma j}\}$.
Conditional on $\{Y,Z\}$,
\begin{align*}
\mu_\gamma(B_{\gamma j})=\frac{1}{N_\gamma}\sum_{i\in\fp}\xi_{\gamma i}A_{\gamma ij},
&&
\E\!\left[\mu_\gamma(B_{\gamma j}) \mid \{Y,Z\}\right]=0,
\end{align*}
and
$$
\Var\left(\mu_\gamma(B_{\gamma j})\mid \{Y,Z\}\right)
=
\frac{1}{N_\gamma^2}\sum_{i\in\fp}\frac{1-\pi_{\gamma i}}{\pi_{\gamma i}}\,A_{\gamma ij}
\le \frac{1}{\neff}.
$$
Moreover, if $\max_i \pi_{\gamma i}^{-1}\le c_0 / \alpha_\gamma$, then for all $t>0$,
\[
\Pr\left(|\mu_\gamma(B_{\gamma j})|>t \mid \{Y,Z\}\right)
\ \le\ 
2\exp\left\{-\frac{t^2}{2\left(v_\gamma + b_\gamma t/3\right\}}\right),
\]
with $v_\gamma\le \neff^{-1}$ and $b_\gamma\le c_0/n_\gamma$.
Under sampling without replacement (rejective), replace $v_\gamma$ by $(1-\alpha_\gamma)\neff^{-1}$.
\end{lemma}

\begin{proof}
The variance identity follows immediately from independence (Poisson--PPS) of $\delta_{\gamma i}$ given $\{Y,Z\}$; the tail bound is Bernstein’s inequality with the stated $v_\gamma,b_\gamma$ (and Lemma~\ref{lem:bernstein-wor} for WOR).
\end{proof}

\begin{proof}[Proof of Proposition~\ref{prop:design-ld}]
By Lemma~\ref{lem:cell-smooth},
\(
\|T_\gamma-\bar f_{\gamma,h}\|_1 \le C_K \sum_{j=1}^{M_\gamma} |\mu_\gamma(B_{\gamma j})| + R_\gamma.
\)
Fix $\tau\in(0,1]$ and set $u:=\tau/(2C_K)$; then
\[
\Pr\left(C_K\sum_{j} |\mu_\gamma(B_{\gamma j})|>\tfrac{\tau}{2} \MID \{Y,Z\}\right)
\le
\Pr\left(\exists j\colon |\mu_\gamma(B_{\gamma j})|>\tfrac{u}{M_\gamma} \Mid \{Y,Z\}\right).
\]
By Lemma~\ref{lem:cell-bernstein} with $t=u/M_\gamma$ and a union bound over $M_\gamma\asymp h_\gamma^{-d}$ cells,
\begin{align*}
\Pr\left(\,\exists j:\ |\mu_\gamma(B_{\gamma j})|>\tfrac{u}{M_\gamma} \Mid \{Y,Z\}\right)
&\le 
2 M_\gamma\exp\left\{-\frac{t^2}{2(v_\gamma + b_\gamma t/3)}\right\} \\
&\le
C \exp\left\{-c \frac{\neff}{M_\gamma}\,\min\{u^2,u\}\right\},
\end{align*}
since $v_\gamma\le n_{\mathrm{eff},\gamma}^{-1}$ and $b_\gamma\le c_0/n_\gamma$, and for large $\gamma$ the $b_\gamma t$ term is dominated by $v_\gamma$ when $t\lesssim M_\gamma^{-1}$.
Because $M_\gamma\asymp h_\gamma^{-d}$, we obtain
\[
\Pr\left(C_K\sum_{j} |\mu_\gamma(B_{\gamma j})|>\tfrac{\tau}{2} \MID \{Y,Z\}\right)
\le
C \exp\left\{-c\,\neff\, h_\gamma^{d}\,\min(\tau^2,\tau)\right\}.
\]
Finally, add the bound for the remainder $\Pr(R_\gamma>\tau/2)\le C e^{-cN_\gamma h_\gamma^d}$ from Lemma~\ref{lem:cell-smooth}.
The WOR version follows by replacing $v_\gamma$ by $(1-\alpha_\gamma)\neff^{-1}$ in Lemma~\ref{lem:cell-bernstein}.
\end{proof}

\subsection{KDE large-deviation bounds}

\begin{lemma}[i.i.d.\ KDE $L^1$ tail]\label{lem:iid-L1}
Under assumptions~\ref{ass:kernel-1} and~\ref{ass:superpopulation}, there exist constants $C,c>0$ (depending only on $K,d$) such that for all $\tau\in(0,1]$,
\begin{equation}\label{eqn:iid-ld}
\Pr\left(\left\|\bar f_{\gamma,h}-g*K_{h_\gamma}\right\|_{L^1}>\tau\right)
\le
C \exp\left\{-c\,N_\gamma h_\gamma^{\,d}\,\min(\tau^2,\tau)\right\}.
\end{equation}
\end{lemma}

\begin{proof}[Proof sketch.]
Partition $\R^d$ into cubes $\{B_{\gamma j}\}_{j=1}^{M_\gamma}$ of side $h_\gamma$ with centers $z_{\gamma j}$, where $M_\gamma\asymp h_\gamma^{-d}$.
For each cell center,
\[
S_{\gamma j}:=\frac{1}{N_\gamma}\sum_{i=1}^{N_\gamma}\left\{K_{h_\gamma}(z_{\gamma j}-Y_{\gamma i}) - \E\left[ K_{h_\gamma}(z_{\gamma j}-Y) \right]\right\}
\]
is a sum of independent centered bounded variables with variance $\lesssim (N_\gamma h_\gamma^d)^{-1}$; Bernstein yields
$\Pr(|S_{\gamma j}|>u) \le 2\exp\{-c\,N_\gamma h_\gamma^d\min(u^2,u)\}$.
A union bound over $M_\gamma\asymp h_\gamma^{-d}$ centers gives
\(
\max_j |S_{\gamma j}| \le u
\)
with probability at least $1 - C \exp\{-c\,N_\gamma h_\gamma^d \min(u^2,u)\}$.
Using the Lipschitz–translation inequality for $K_{h_\gamma}$, which is valid since $K\in W^{1,1}$,
\[
\int_{B_{\gamma j}}\left|(\bar f_{\gamma,h} - g * K_{h_\gamma})(y) - S_{\gamma j}\right| \dx y
\le C_K h_\gamma^d.
\]
Summing over $j$ yields
\(
\|\bar f_{\gamma,h}- g * K_{h_\gamma}\|_{L^1}
\le
M_\gamma h_\gamma^d \max_j |S_{\gamma j}| + C_K M_\gamma h_\gamma^d
= \max_j |S_{\gamma j}| + C_K'.
\)
Choosing $u\simeq \tau/2$ and noticing that $C_K'$ can be absorbed into the $\min(\tau^2,\tau)$ regime for $\tau\in(0,1]$ gives the desired bound~\eqref{eqn:iid-ld}.
A full proof parallels the one of Proposition~\ref{prop:design-ld}, with independence replacing design weighting.
\end{proof}

\begin{lemma}[Approximate identity]\label{lem:approx-id}
If $K\in L^1(\R^d)$ with $\int K=1$, then for every $g\in L^1(\R^d)$,
\(
\|g*K_{h_\gamma}-g\|_{L^1}\to 0
\)
as $h_\gamma\downarrow 0$.
\end{lemma}

\subsection{Consistency of the HT-adjusted KDE}

\begin{proof}[Proof of Theorem~\ref{thm:ht-kde-consistency}]

The proof uses several technical lemmas listed in Appendix~\ref{sec:appendix-lemmas-consistency-kde}.
With the definition of $\bar f_{\gamma,h}(y)$ from Proposition~\ref{prop:design-ld} and using Lemma~\ref{lem:threeWay} we get
\begin{equation}\label{eq:three-way}
\|\hat f_{\gamma}-g\|_1
\le
\underbrace{|S_\gamma^{-1}-1|}_{A_\gamma}
+
\underbrace{\|T_\gamma-\bar f_{\gamma,h}\|_1}_{B_\gamma}
+
\underbrace{\|\bar f_{\gamma,h}-g\|_1}_{C_\gamma}.
\end{equation}

\paragraph*{HT normalizer concentration.}

Let \(X_{\gamma i}:=N_\gamma^{-1}(\delta_{\gamma i}/\pi_{\gamma i}-1)\).
Conditional on \(Z\), the \(X_{\gamma i}\) are independent, mean zero, and \(|X_{\gamma i}|\le c_0/n_\gamma\) by Assumption~\ref{ass:design-regularity}.
Moreover
\[
\Var(S_\gamma\mid Z)=\frac{1}{\tneff}
\le
\frac{1}{\neff}.
\]
Bernstein bounds from Lemmas~\ref{lem:bernstein-ind}--\ref{lem:bernstein-wor} yield constants \(c,c'>0\) with
\[
\Pr\left(|S_\gamma-1| > t\ \mid\ Z\right)
\le
2\exp\left\{-\frac{c\neff t^2}{1+c' t}\right\}.
\]
For \(|S_\gamma-1|\le 1/2\), \(A_\gamma \le 2|S_\gamma-1|\), hence \(A_\gamma=o_\Pr(1)\) with exponential tails.

\paragraph*{Design noise in the numerator.}
Proposition~\ref{prop:design-ld} shows there exist constants $c,C>0$ such that
$$
\Pr\left(B_\gamma > \tau \mid \{Y,Z\}\right)
\le
C\exp\left\{-c\neff h_\gamma^{d}\,\min(\tau^2,\tau)\right\}
+
C\exp\left\{-c N_\gamma h_\gamma^{d}\right\}.
$$

\paragraph*{Smoothing noise and bias.}
Decompose
$$
C_\gamma
\le
\left\|\bar f_{\gamma,h}- g * K_{h_\gamma}\right\|_{L^1}
+
\left\|g * K_{h_\gamma} - g\right\|_{L^1}.
$$
Then, combining Lemmas~\ref{lem:iid-L1} and~\ref{lem:approx-id} shows that $C_\gamma \to 0$ in probability, with the tail
$$
\Pr\left(\left\|\bar f_{\gamma,h}-g*K_{h_\gamma}\right\|_{L^1}>\tau\right)
\le
C\exp\left\{-c\,N_\gamma\, h_\gamma^{d}\,\min(\tau^2,\tau)\right\}.
$$
for the stochastic part.

\paragraph*{Conclusion}
Plugging the three steps above into the three-way decomposition \eqref{eq:three-way} we obtain
$$
\Pr\left(\|\hat f_{\gamma}-g\|_{L^1}>\tau\right)
\le
C \exp\left\{-c\,\neff\,h_\gamma^{d}\,\min(\tau^2,\tau)\right\}
+
C \exp\left\{-c\,N_\gamma h_\gamma^{d}\right\}
+
o(1),
$$
and therefore $\|\hat f_{\gamma}-g\|_{L^1}\xrightarrow{\,p\,}0$ as $\gamma\to\infty$ whenever $\neff\, h_\gamma^{d}\to\infty$.
\end{proof}

\begin{remark}[On the $N_\gamma h_\gamma^d$ growth]
Note that
\[
\neff
=\frac{N_\gamma^2}{\sum_{i\in\fp}\pi_{\gamma i}^{-1}}
\le
\frac{N_\gamma^2}{N_\gamma^2/\sum_{i\in\fp} \pi_{\gamma i}}
=
\sum_{i\in\fp} \pi_{\gamma i}
=:
n_\gamma
\le
N_\gamma,
\]
so $N_\gamma h_\gamma^d \ge \neff\, h_\gamma^d$ for all $\gamma$.
Therefore, the condition $\neff\,h_\gamma^d\to\infty$ implies $N_\gamma h_\gamma^d\to\infty$, without any additional assumptions on $\alpha_\gamma$.
\end{remark}

\subsection{Consistency of the MHDE}\label{sec:appendix-proof-consistency-mhde}

\begin{lemma}[Uniform Hellinger control]\label{lem:uniform}
For all $\gamma$,
$$
\sup_{\theta\in\Theta}\left|\Gamma_\gamma(\theta)-\Gamma(\theta)\right|
\le
\left\|\sqrt{\hat f_{\gamma}}-\sqrt{g}\right\|_{L^2}.
$$
\end{lemma}

\begin{proof}
For any $\theta$, by Cauchy-Schwarz,
$$
|\Gamma_\gamma(\theta)-\Gamma(\theta)|
=\left|\int \sqrt{f_\theta}\,\left(\sqrt{\hat f_{\gamma}}-\sqrt g\right)\right|
\le
\left\|\sqrt{f_\theta}\right\|_2\,\left\|\sqrt{\hat f_{\gamma}}-\sqrt g\right\|_2
= \left\|\sqrt{\hat f_{\gamma}}-\sqrt g\right\|_2,
$$
since $\|\sqrt{f_\theta}\|_2=(\int f_\theta)^{1/2}=1$.
Taking the supremum over $\theta$ gives the claim.
\end{proof}

\begin{lemma}[Hellinger vs.\ $L_1$]\label{lem:HvsL1}
For densities $p,q$ on $\R^d$,
$$
\|\sqrt p-\sqrt q\|_{L^2}^2 \ \le\ \|p-q\|_{L^1}.
$$
\end{lemma}

\begin{proof}
Pointwise for $a,b\ge 0$ and w.l.o.g.\ $a\ge b$,
$(\sqrt a-\sqrt b)^2=(a-b)/(\sqrt a+\sqrt b)\le a-b$.
Integrate with $a=p(y)$ and $b=q(y)$.
\end{proof}

\begin{proof}[Proof of Proposition~\ref{prop:mhde-tail}]
By Lemmas~\ref{lem:uniform} and~\ref{lem:HvsL1},
$$
\sup_\theta|\Gamma_\gamma(\theta)-\Gamma(\theta)|
\le
\left\|\sqrt{\hat f_{\gamma}}-\sqrt g\right\|_2
\le
\left\|\hat f_{\gamma}-g\right\|_1^{1/2}.
$$
Hence, for $t\in(0,1]$,
$$
\Pr\left(\sup_\theta|\Gamma_\gamma(\theta)-\Gamma(\theta)|>t\right)
\le
\Pr\left(\|\hat f_{\gamma}-g\|_1>t^2\right).
$$

Apply the $L_1$ large-deviation bound of Theorem~\ref{thm:ht-kde-consistency} with $\tau=t^2$, which yields $\exp\{-c\,\neff\,h_\gamma^d\min(t^4,t^2)\}$ for the design term and $\exp\{-c\,N_\gamma h_\gamma^d\}$ for the i.i.d.\ smoothing term.
The WOR factor $(1-\alpha_\gamma)$ enters as in Theorem~\ref{thm:ht-kde-consistency}.
\end{proof}

\section{Proof of the CLT}\label{sec:proof-clt}

We first define the notation used throughout this proof.
For any distribution $H$ with density $h$ we write

\begin{align*}
\nabla_\theta \Gamma_h(\theta)\MID_{\theta=\theta_0}
=\int \phi_g(y) \left(h(y) - g(y) \right) R_h(y)\dx y,
&&
R_h(y):=\frac{2\sqrt{g(y)}}{\sqrt{h(y)}+\sqrt{g(y)}}\in[0,2].
\end{align*}
Since $\theta_0$ maximizes $\Gamma_G(\theta)=\int\sqrt{f_\theta}\sqrt{g}$, we have $\nabla_\theta \Gamma_G(\theta_0)=0$.

Let's further define $\psi_{h_\gamma} := K_{h_\gamma} * \psi_g$.
We continue to use the notation $\widehat f_{\gamma}=S_\gamma^{-1}T_\gamma$ for the HT-adjusted KDE and $\bar f_{\gamma,h}(y)=$ for the unweighted KDE as in Proposition~\ref{prop:design-ld}.

\paragraph*{Algebraic decomposition.}
Add and subtract $\bar f_{\gamma,h}$ and $g*K_{h_\gamma}$, and isolate the normalizer:
\begin{align*}
\nabla_\theta\Gamma_\gamma(\theta_0)
&=\int \phi_g\,(S_\gamma^{-1}T_\gamma-g)\,R_\gamma
=\underbrace{\int \phi_g\,(T_\gamma-\bar f_{\gamma,h})\,R_\gamma}_{\mat A_{\gamma,1}}
+\underbrace{\int \phi_g\,(\bar f_{\gamma,h}-g*K_{h_\gamma})\,R_\gamma}_{\mat A_{\gamma,2}}\\
&\quad+\underbrace{\int \phi_g\,(g*K_{h_\gamma}-g)\,R_\gamma}_{\mat B_\gamma}
+\underbrace{(S_\gamma^{-1}-1)\int \phi_g\,T_\gamma\,R_\gamma}_{\mat N_\gamma^{(S)}}.
\end{align*}
Split $A_{\gamma,1}$ and $A_{\gamma,2}$ into a linear piece and a nonlinear remainder involving $R_\gamma-1$:
$$
\mat A_{\gamma,1}=\underbrace{\int \phi_g\,(T_\gamma-\bar f_{\gamma,h})}_{\mat \Xi_\gamma} 
+\underbrace{\int \phi_g\,(T_\gamma-\bar f_{\gamma,h})\,(R_\gamma-1)}_{\mat R_{\gamma,2}^{(1)}},
$$
$$
\mat A_{\gamma,2}=\underbrace{\int \phi_g\,(\bar f_{\gamma,h}-g*K_{h_\gamma})}_{\mat U_\gamma}
+\underbrace{\int \phi_g\,(\bar f_{\gamma,h}-g*K_{h_\gamma})\,(\mat R_\gamma-1)}_{\mat R_{\gamma,2}^{(2)}}.
$$
Set $\mat R_{\gamma,2}:=\mat R_{\gamma,2}^{(1)}+\mat R_{\gamma,2}^{(2)}$.

\paragraph*{Identify the HT fluctuation.}
By Fubini,
$$
\mat \Xi_\gamma
=\int \phi_g (T_\gamma-\bar f_{\gamma,h})
=\frac{1}{N_\gamma}\sum_{i \in \fp}\left(\frac{\delta_{\gamma i}}{\pi_{\gamma i}}-1\right) \psi_{h_\gamma}(Y_{\gamma i}),
$$
with $\psi_{h_\gamma}=K_{h_\gamma}*\phi_g$.

\paragraph*{Variance limit and CLT for $\mat \Xi_\gamma$.}
Define
$$
\mat V_\gamma:=\tneff \Var \left(\mat \Xi_\gamma \Mid \{Y_{\gamma i}\}\right)
=\tneff \frac{1}{N_\gamma^2}\sum_{i \in \fp}\frac{1-\pi_{\gamma i}}{\pi_{\gamma i}} 
\psi_{h_\gamma}(Y_{\gamma i})\psi_{h_\gamma}(Y_{\gamma i})\tr.
$$
Let
$
w_{\gamma i}:=\dfrac{(1-\pi_{\gamma i})/\pi_{\gamma i}}{\sum_{j \in \fp}(1-\pi_{\gamma j})/\pi_{\gamma j}}
$
so that $\mat V_\gamma=\sum_{i \in \fp} w_{\gamma i} \psi_{h_\gamma}(Y_{\gamma i})\psi_{h_\gamma}(Y_{\gamma i})^\top$ and $\max_i w_{\gamma i}\to 0$ by the ``no dominant unit'' in Assumption~\ref{ass:clt-lindeberg}.
By Lemma~\ref{lem:clt-weighted-lln} (applied to each coordinate) and $\E_g\|\psi_{h_\gamma}(Y)\|<\infty$ (from $\phi_g\in L^2(g)$ and $K_{h_\gamma}\in L^1$),
$$
\mat V_\gamma \xrightarrow{\,\Pr\,} \E_g\left[\psi_{h_\gamma}(Y)\psi_{h_\gamma}(Y)^\top\right]=: \mat \Sigma_{h_\gamma}.
$$
By Lemma~\ref{lem:clt-approx-identity},
$$
\|\psi_{h_\gamma}-\phi_g\|_{L^2(g)}\to 0
\quad\Rightarrow\quad
\mat\Sigma_{h_\gamma}\to \mat\Sigma:=\E_g[\phi_g\phi_g^\top]
$$
(since $\|ab\tr-ac\tr\|_{\mathrm{HS}}\le \|a\|_2 \|b-c\|_2+\|c\|_2 \|a-b\|_2$).
Finally, the triangular-array Lindeberg condition in Assumption~\ref{ass:clt-lindeberg} yields, conditionally on $\{Y_{\gamma i}\}$,
$$
\sqrt{\tneff} \mat \Xi_\gamma \xrightarrow{\;} \mathcal N_p(0,\mat\Sigma),
$$
and unconditioning preserves the limit.

\paragraph*{i.i.d.\ smoothing term $U_\gamma$.}
Write
\begin{align*}
\mat U_\gamma
&= \int \phi_g(y) \left[\frac{1}{N_\gamma}\sum_{i=1}^{N_\gamma}\left\{K_{h_\gamma}(y-Y_{\gamma i})-\E K_{h_\gamma}(y-Y)\right\}\right]\dx y \\
&=\frac{1}{N_\gamma}\sum_{i=1}^{N_\gamma}\left\{\psi^\sharp_{h_\gamma}(Y_{\gamma i})-\E \psi^\sharp_{h_\gamma}(Y)\right\},
\end{align*}
where $\psi^\sharp_{h_\gamma}:=K^\sim_{h_\gamma}*\phi_g$ and $K^\sim(t):=K(-t)$.
Hence $\E[\mat U_\gamma]=0$ and
\begin{align*}
\Var(\mat U_\gamma)
&=\frac{1}{N_\gamma} \E_g\left[\psi^\sharp_{h_\gamma}(Y)^2\right]
=\frac{1}{N_\gamma} \|K_{h_\gamma}*\phi_g\|_{L^2(g)}^2 \\
& \le\frac{\|\phi_g\|_{L^2(g)}^2}{N_\gamma}\int \frac{(K_{h_\gamma}^2*g)(y)}{g(y)}\dx y,
\end{align*}
by Lemma~\ref{lem:clt-conv-weight}.
Under the localized risk Assumption~\ref{ass:clt-local-risk},
$
\int (K_{h_\gamma}^2*g)/g \le C_R h_\gamma^{-d} + \tau_R(h_\gamma)
$
for any fixed $R$, uniformly for small $h_\gamma$.
Choosing $R=R_\gamma\uparrow\infty$ with $\tau_{R_\gamma}(h_\gamma)\to 0$ yields
$$
\Var(\mat U_\gamma)\lesssim\frac{1}{N_\gamma h_\gamma^d}\quad\Rightarrow\quad
\sqrt{\tneff} \mat U_\gamma\xrightarrow{\;\Pr\;}0
$$
by Assumption~\ref{ass:clt-bandwidth}.

\paragraph*{Bias term $\mat B_\gamma$.}
By Cauchy--Schwarz and $R_\gamma\in[0,2]$,
$$
|\mat B_\gamma|
\le 2 \|\phi_g\|_{L^2(g)} \left\|\frac{g*K_{h_\gamma}-g}{\sqrt g}\right\|_{L^2}.
$$
Under either kernel route in Assumption~\ref{ass:clt-kernel} and the localized bias bound in Assumption~\ref{ass:clt-local-risk},
$
\|(g*K_{h_\gamma}-g)/\sqrt g\|_{2}\lesssim h_\gamma^\beta,
$
hence
$
\sqrt{\tneff} \mat B_\gamma\to 0
$
by Assumption~\ref{ass:clt-bandwidth}.

\paragraph*{Nonlinear remainder $R_{\gamma,2}$.}
Apply Lemma~\ref{lem:clt-rem-hellinger-chisq} with $u=T_\gamma-\bar f_{\gamma,h}$ and $u=\bar f_{\gamma,h}-g*K_{h_\gamma}$, and $\varphi=\phi_g$:
$$
|\mat R_{\gamma,2}|
 \le (M H(\widehat f_{HT,\gamma},g)+\sqrt{\epsilon(M)})\left(
\left\|\frac{T_\gamma-\bar f_{\gamma,h}}{\sqrt g}\right\|_{2}
+\left\|\frac{\bar f_{\gamma,h}-g*K_{h_\gamma}}{\sqrt g}\right\|_{2}
\right).
$$
Under the localized risk bounds,
\begin{align*}
\E\left\|\frac{T_\gamma-\bar f_{\gamma,h}}{\sqrt g}\right\|_{2}^2  & \lesssim \frac{1}{\tneff h_\gamma^d}\quad
\text{(Poisson--PPS; WOR has FPC $(1-\alpha_\gamma)$),}\\
\E\left\|\frac{\bar f_{\gamma,h}-g*K_{h_\gamma}}{\sqrt g}\right\|_{2}^2  &\lesssim \frac{1}{N_\gamma h_\gamma^d}.
\end{align*}
Moreover $H(\widehat f_{\gamma},g)=o_\Pr(\tneff^{-1/2})$ under Assumptions~\ref{ass:clt-bandwidth} and~\ref{ass:clt-design}.
Choose $M=M_\gamma\uparrow\infty$ so that $\epsilon(M_\gamma)\to 0$.
Then
$
\sqrt{\tneff} \mat R_{\gamma,2}\xrightarrow{\Pp}0.
$

\paragraph*{Self-normalizer $\mat N_\gamma^{(S)}$.}
Expand $S_\gamma^{-1}=1-(S_\gamma-1)+\rho_\gamma$, where $|\rho_\gamma|\le 2(S_\gamma-1)^2$ on $\{|S_\gamma-1|\le \tfrac12\}$, which holds with high probability by the Bernstein/Freedman bound in Assumption~\ref{ass:clt-design}.
Then,
\begin{align*}
\mat N_\gamma^{(S)}=-(S_\gamma-1) \mat M_\gamma + \rho_\gamma \mat M_\gamma,
&&
\mat M_\gamma:=\frac{1}{N_\gamma}\sum_{i\in\fp}\frac{\delta_{\gamma i}}{\pi_{\gamma i}}\tilde\psi_{h_\gamma,\gamma}(Y_{\gamma i}),
\end{align*}
with $\tilde\psi_{h_\gamma,\gamma}(y):=\int \phi_g(x)K_{h_\gamma}(x-y)R_\gamma(x) dx$.
Decompose
$$
\mat M_\gamma=\underbrace{\frac{1}{N_\gamma}\sum_{i\in\fp} \tilde\psi_{h_\gamma,\gamma}(Y_{\gamma i})}_{\mat m_\gamma} 
+\underbrace{\frac{1}{N_\gamma}\sum_{i\in\fp} (\delta_{\gamma i}/\pi_{\gamma i}-1)\tilde\psi_{h_\gamma,\gamma}(Y_{\gamma i})}_{\text{HT fluctuation}}.
$$
By Lemma~\ref{lem:clt-approx-identity} and $R_\gamma\to 1$ in $L^1(g)$, $\mat m_\gamma-\E_g[\psi_{h_\gamma}(Y)]\to 0$ in probability, and $\E_g[\psi_{h_\gamma}(Y)]=\int\phi_g g=0$ (because $\nabla \Gamma_G(\theta_0)=0$).
Hence $\mat m_\gamma=o_\Pr(1)$ and the HT fluctuation is $O_\Pr((\tneff h_\gamma^d)^{-1/2})$.
Using $\sqrt{\tneff}(S_\gamma-1)=O_\Pr(1)$ and $\sqrt{\tneff} \rho_\gamma=O_\Pr(\tneff^{-1/2})$, we get $\sqrt{\tneff} \mat N_\gamma^{(S)}\to 0$ in probability.

\paragraph*{CLT conclusion via Taylor expansion.}
Collecting all the previous steps,
$$
\sqrt{\tneff} \nabla_\theta\Gamma_\gamma(\theta_0)
=\sqrt{\tneff} \mat\Xi_\gamma+o_\Pr(1)
 \xrightarrow{\quad} \mathcal N_p(0, \mat \Sigma).
$$
(For fixed-size sampling, multiply by $(1-\alpha)$.)
A second-order Taylor expansion around $\theta_0$ gives
\(
0=\nabla_\theta\Gamma_\gamma(\hat\theta_\gamma)
=\nabla_\theta\Gamma_\gamma(\theta_0)-\mat A_\gamma(\hat\theta_\gamma-\theta_0)+r_\gamma,
\)
with $\mat A_\gamma:=-\nabla_\theta^2\Gamma_\gamma(\tilde\theta_\gamma)\xrightarrow{\,\Pr\,} \mat A$ (dominated differentiation, uniform LLN under the localized risk) and $r_\gamma=o_\Pr(\|\hat\theta_\gamma-\theta_0\|)$.
Thus
$$
\sqrt{\tneff} (\hat\theta_\gamma-\theta_0)
=\mat A_\gamma^{-1} \sqrt{\tneff} \nabla_\theta\Gamma_\gamma(\theta_0)+o_\Pp(1)
 \xrightarrow{\quad} \mathcal N_p \left(0, \mat A^{-1} \mat\Sigma \mat A^{-\intercal}\right)
$$
for Poisson--PPS, and with covariance $\mat A^{-1}[(1-\alpha)\mat\Sigma] \mat A^{-\intercal}$ for fixed-size SRS--WOR.

\qed

\subsection{Technical Lemmas}

\begin{lemma}[Approximate identity in $L^2(f)$ via localization]\label{lem:clt-approx-identity}
Under assumptions~\ref{ass:kernel-1} and~\ref{ass:clt-bandwidth}, let $A_R\uparrow \R^d$ be compacts with $0<c_R\le g\le C_R<\infty$ on $A_R$.
If $\phi_g\in L^2(g)$ then $\|K_{h_\gamma}*\phi_g-\phi_g\|_{L^2(g)}\to 0$ as $\gamma\to\infty$.
\end{lemma}

\begin{proof}
On $A_R$, $g$ is bounded above and below, hence
$
\int_{A_R}\!\!|K_{h_\gamma}*\phi_g-\phi_g|^2 g
\le C_R\|K_{h_\gamma}*\phi_g-\phi_g\|_{L^2}^2\to 0
$
by the $L^2$ approximate identity (Young + density).
On $A_R^c$,
$
\int_{A_R^c}\!\!|K_{h_\gamma}*\phi_g-\phi_g|^2 g
\le 2\int_{A_R^c}\!\!|K_{h_\gamma}*\phi_g|^2 g + 2\int_{A_R^c}\!\!|\phi_g|^2 g.
$
The second term goes to $0$ from above as $R\uparrow\infty$ since $\phi_g\in L^2(g)$.
For the first term, fix $R$ and split $\phi_g=\phi_g\mathbf{1}_{A_R}+\phi_g\mathbf{1}_{A_R^c}$, use $K_{h_\gamma}\in L^1$, Young, and the previous tail smallness of $\phi_g\mathbf{1}_{A_R^c}$ to make it $<\varepsilon$ uniformly for small $h_\gamma$.
Now take $\gamma\to\infty$, then $R\to\infty$.
\end{proof}

\begin{lemma}[Weighted LLN for triangular weights]\label{lem:clt-weighted-lln}
Let $Z_{\gamma i}\in\R^q$ be i.i.d.\ with $\E\|Z_{\gamma 1}\|<\infty$.
Let $w_{\gamma i}\ge 0$ with $\sum_{i \in\fp} w_{\gamma i}=1$ and $\max_{i \in\fp} w_{\gamma i}\to 0$.
Then $\sum_{i\in\fp} w_{\gamma i} Z_{\gamma i}\xrightarrow{\Pp}\E Z_{\gamma 1}$.
\end{lemma}

\begin{proof}
Write the scalar case and apply Chebyshev with $\Var(\sum_{i \in\fp} w_i Z_i)\le (\max_i w_i)\sum w_i \E\|Z_i-\E Z\|^2$ when $\E\|Z\|^2<\infty$ (obtainable by truncation if only the first moment exists).
Extend to vectors by component-wise application.
\end{proof}

\begin{lemma}[Weighted convolution inequality]\label{lem:clt-conv-weight}
For any $\varphi\in L^2(g)$,
$$
\|K_{h_\gamma}*\varphi\|_{L^2(g)}^2
\le\|\varphi\|_{L^2(g)}^2\int_{\R^d}\frac{(K_{h_\gamma}^2*g)(y)}{g(y)}\dx y.
$$
\end{lemma}
\begin{proof}
By Cauchy--Schwarz with weight $g(x)$ inside the convolution:
\begin{align*}
\Bigg|(K_{h_\gamma}*\varphi)(y)\Bigg|
&= \left|\int \varphi(x)\,K_{h_\gamma}(x-y)\dx x\right|
=\left|\int \varphi(x)\sqrt{g(x)}\frac{K_{h_\gamma}(x-y)}{\sqrt{g(x)}}\dx x\right| \\
&\le \|\varphi\|_{L^2(g)}\left(\int \frac{K_{h_\gamma}^2(x-y)}{g(x)}\dx x\right)^{1/2}.
\end{align*}
Taking the square and multiplying by $g(y)$, then integrating over $y$ to obtain the claim after Fubini.
\end{proof}

\begin{lemma}[Remainder via Hellinger and $\chi^2(g)$]\label{lem:clt-rem-hellinger-chisq}
For any $\varphi\in L^2(g)$, any square-integrable $u$, and densities $h,g$,
$$
\left|\int \varphi u (R_h-1)\right|
 \le M H(h,g) \left\|\frac{u}{\sqrt g}\right\|_{2} + \sqrt{\epsilon(M)} \left\|\frac{u}{\sqrt g}\right\|_{2},
$$
where $R_h=2\sqrt g/(\sqrt h+\sqrt g)$, $H(h,g)=\|\sqrt h-\sqrt g\|_2$, and $\epsilon(M):=\int_{\{|\varphi|>M\}}\varphi^2 g dy$.
\end{lemma}
\begin{proof}
Split the domain into $\{|\varphi|\le M\}$ and its complement.
On $\{|\varphi|\le M\}$,
$$
\int |\varphi u|\,|R_h-1|
\le M\int \frac{|u|}{\sqrt g} \frac{|\sqrt h-\sqrt g|}{\sqrt h+\sqrt g} \sqrt g
\le M\left\|\frac{u}{\sqrt g}\right\|_2 \left(\int g |R_h-1|^2\right)^{1/2}.
$$
Since $|R_h-1|={|\sqrt h-\sqrt g|}/{(\sqrt h+\sqrt g)}$ and $\sqrt h+\sqrt g\ge \sqrt g$,
$
\int g |R_h-1|^2\le \int (\sqrt h-\sqrt g)^2=H(h,g)^2.
$
On $\{|\varphi|>M\}$, Cauchy--Schwarz gives
$
\int |\varphi u||R_h-1|\le \sqrt{\epsilon(M)} \|u/\sqrt g\|_2
$
since $|R_h-1|\le 1$.
Combine the two bounds.
\end{proof}

\section{Robustness proof}\label{sec:proof-robust}

\begin{lemma}[Uniform Hellinger inequality]\label{lem:hell-unif}
For any distributions $F_1, F_2$ with densities $f_1, f_2$ and any $\theta$,
$$
\Bigg| \Gamma_{F_1}(\theta)-\Gamma_{F_2}(\theta) \Bigg|
= \left|\int \sqrt{f_\theta(y)} \left(\sqrt{f_1(y)}-\sqrt{f_2(y)}\right) \dx y \right|
\le H(f_1,f_2).
$$
\end{lemma}
\begin{proof}
Cauchy--Schwarz and $\|\sqrt{f_\theta}\|_2=1$.
\end{proof}

\begin{proposition}[Continuity of $T$ in the Hellinger topology.]\label{prop:continuity}
Let $(G_n)$ be a sequence of distributions with densities $g_n$ s.t.\ $H(g_n, g)\to 0$.
Under Assumption~\ref{ass:infl-model}, any sequence of maximizers $T(G_n)$ satisfies $T(G_n)\to \theta_0=T(G)$.
\end{proposition}

\begin{proof}
By Lemma~\ref{lem:hell-unif}, $\sup_\theta|\Gamma_{g_n}(\theta)-\Gamma_g(\theta)|\le H(g_n,g)\to 0$.
Fix $\varepsilon>0$.
By separation, $\sup_{\|\theta-\theta_0\|\ge\varepsilon}\Gamma_g(\theta)\le \Gamma_g(\theta_0)-\Delta(\varepsilon)$.
For $n$ large with $H(g_n,g)\le \Delta(\varepsilon)/3$,
\begin{align*}
\sup_{\|\theta-\theta_0\|\ge\varepsilon}\Gamma_{g_n}(\theta)
\le \Gamma_g(\theta_0)-\tfrac{2}{3}\Delta(\varepsilon),
&&
\Gamma_{g_n}(\theta_0)\ \ge\ \Gamma_g(\theta_0)-\tfrac{1}{3}\Delta(\varepsilon).
\end{align*}
Hence any maximizer $T(G_n)$ must lie in $B(\theta_0,\varepsilon)$.
Since $\varepsilon$ is arbitrary, $T(G_n)\to\theta_0$.
\end{proof}

\begin{proof}[Proof of Theorem~\ref{thm:influence}]
Let $G_\varepsilon=(1-\varepsilon)G+\varepsilon H$ and $\theta_\varepsilon:=T(G_\varepsilon)$.
By definition, $S_{G_\varepsilon}(\theta_\varepsilon)=0$ for small $\varepsilon$.
Taylor expanding $S_{G_\varepsilon}(\theta)$ at $(\theta_0,G)$,
\[
0=S_G(\theta_0) + \mat Q (\theta_\varepsilon-\theta_0)+\dot S_G(\theta_0;H) \varepsilon + R_\varepsilon,
\]
where $\|R_\varepsilon\|=o(\|\theta_\varepsilon-\theta_0\|)+o(\varepsilon)$ by the dominated smoothness in Assumption~\ref{ass:infl-model}(ii) and the definition of $\dot S_G$.
Since $S_G(\theta_0)=0$ and $\mat Q$ is nonsingular,
$$
\theta_\varepsilon-\theta_0 = -\mat Q^{-1} \dot S_G(\theta_0;H) \varepsilon + o(\varepsilon).
$$
Dividing by $\varepsilon$ and letting $\varepsilon\downarrow 0$ gives the stated derivative.
\end{proof}

\section{Additional simulation results}\label{sec:suppl-simres}

Here we present additional results for the simulations in Section~4.1 of the main manuscript.
Unless otherwise noted, the simulation settings are identical to the main manuscript.

\subsection{Gamma Model with Calibrated Weights}\label{sec:suppl-gamma-calibrated}

We now consider that each unit in the finite population $\mathscr U$ belongs to one of five clusters.
We then consider calibrated sampling weights based on an auxiliary variable $X$, with known cluster totals.

Specifically, for each $i \in \mathscr U$ we know the cluster assignment via the membership function $C \colon i\to\{1,\dotsc,5\}$.
Moreover, we assume to know the cluster totals for the population, $\overline X_c = \sum_{i\in\mathscr U} X_i \mathbf{1}_{\{ C(i) = c \}}$, $c \in \{1,\dotsc,5\}$.
Given a sample $\mathscr S$ and survey weights $w_i$, we determine the calibration adjustment factors $\xi_c = \frac{1}{\overline X_c} \sum_{i\in\mathscr S} w_i X_i \mathbf{1}_{\{ C(i) = c \}}$.
The calibrated weights are then given by $\omega_i = w_i \xi_{C(i)}$ for $i \in \mathscr S$.

The results for calibrated weights are very similar to the results with the original survey weights.
The relative bias in Figure~\ref{fig:sim-gamma-rel-bias-calibrated} and the relative variance in Figure~\ref{fig:sim-gamma-rel-rmse-calibrated} still show that the MHDE is very close to the MLE across all scenarios.

\begin{figure}[t]
    \centering
    \includegraphics[width=1\linewidth]{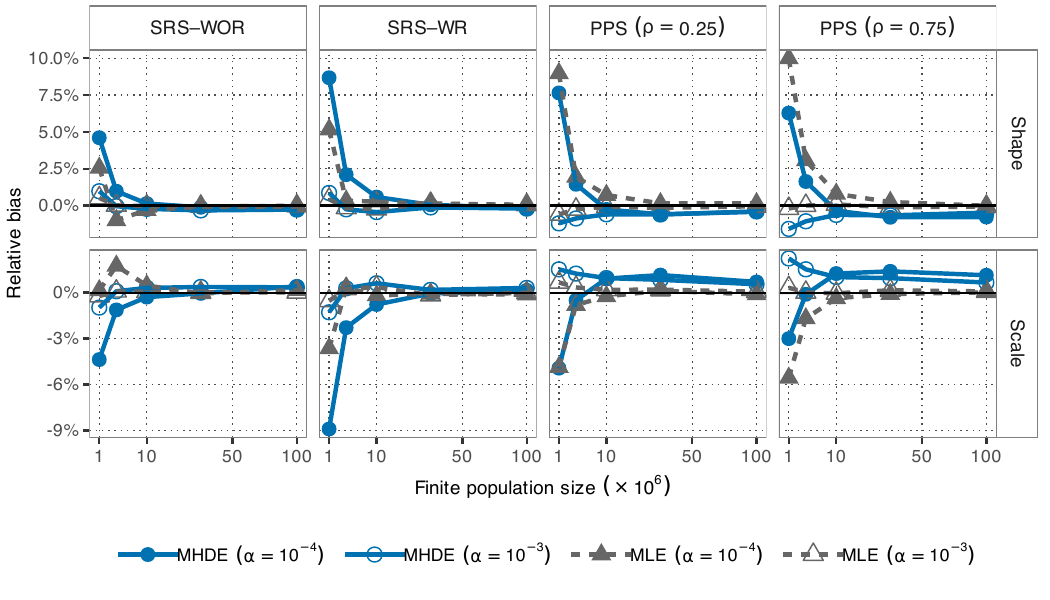}
    \caption{%
      Relative bias of the MHDE (blue dots) and MLE (gray triangles) in a Gamma superpopulation model using calibrated weights and various sampling designs.
      The sample size in each simulation is determined by $n=\alpha N$, with $\alpha \in \{10^{-3},10^{-4}\}$.%
      }
    \label{fig:sim-gamma-rel-bias-calibrated}
\end{figure}

\begin{figure}[t]
    \centering
    \includegraphics[width=1\linewidth]{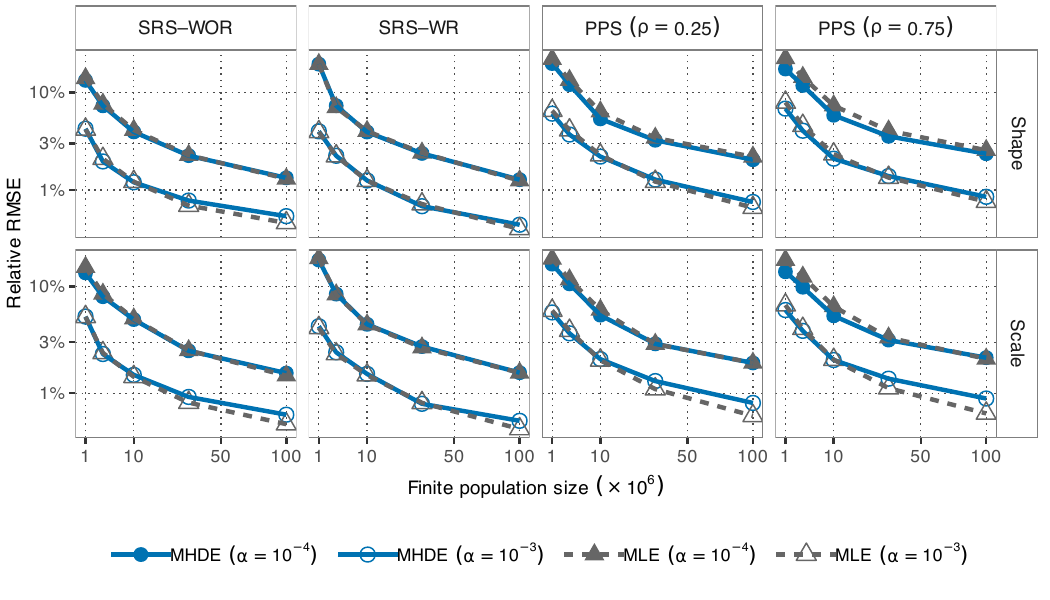}
    \caption{%
      Relative RMSE of the MHDE (blue dots) and MLE (gray triangles) under a Gamma superpopulation model using calibrated weights and various sampling designs.
      The sample size in each simulation is determined by $n=\alpha N$, with $\alpha \in \{10^{-3},10^{-4}\}$.%
      }
    \label{fig:sim-gamma-rel-rmse-calibrated}
\end{figure}

\subsubsection{Alternative Contamination Model}\label{sec:suppl-simres-tcont}

Here we consider a truncated \textit{t} distribution as the source of contamination instead of the Normal distribution from the main manuscript.
We replace $\lfloor \varepsilon n \rfloor$ observations in the sample with i.i.d.~draws from a shifted and truncated (positive) \textit{t} distribution with 3 degrees of freedom with mode at $z > 0$.
We further scale the contamination to have the same variance as the true Gamma distribution from the superpopulation.

Figure~\ref{fig:sim-gamma-infl-tdist} shows the influence function (top) for 10\% contamination and varying 

\begin{figure}[t]
    \centering
    \includegraphics[width=1\linewidth]{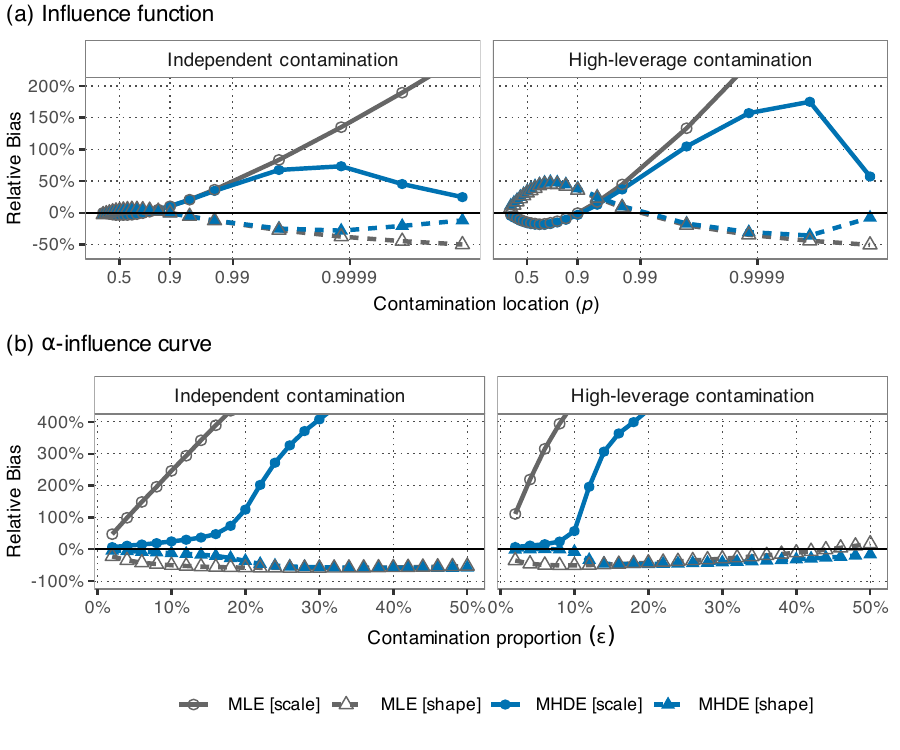}
    \caption{%
      Influence functions (top) and alpha curves (bottom) for the MHDE and MLE of the scale ($\circ$) and shape ($\triangle$) parameters in the Gamma model. %
      The contamination proportion in the influence function at the top is set to $\varepsilon=0.1$. 
      The horizontal axis shows the location of the mode of the truncated \textit{t} distribution in terms of the quantile of the true superpopulation model, i.e., $z = G^{-1}(p)$.%
      For the alpha curve the mode of the \textit{t} distribution is located at $z=G^{-1}(p=0.99) \approx 232\,342$.%
      }
    \label{fig:sim-gamma-infl-tdist}
\end{figure}

\subsection{Lognormal Model}\label{sec:suppl-simres-lognormal}

Instead of the Gamma model, we now consider a lognormal superpopulation model, $Y \sim \text{LN}(\mu=9, \sigma=2)$.
The bias, shown in Figure~\ref{fig:sim-lognormal-rel-bias}, is close to 0 for both the mean and the SD parameters.
Similarly, the variance goes to 0 quickly as the finite population size and the sample size increase, with practically no difference between the MHDE and the MLE.
Due to the minimal bias, inference using the asymptotic distribution of the MHDE is also highly reliable.
The empirical coverage probability of the 95\% CI in Figure~\ref{fig:sim-lognormal-inference} is very close to the nominal level across all sampling schemes, even for small sample sizes.

\begin{figure}[t]
    \centering
    \includegraphics[width=1\linewidth]{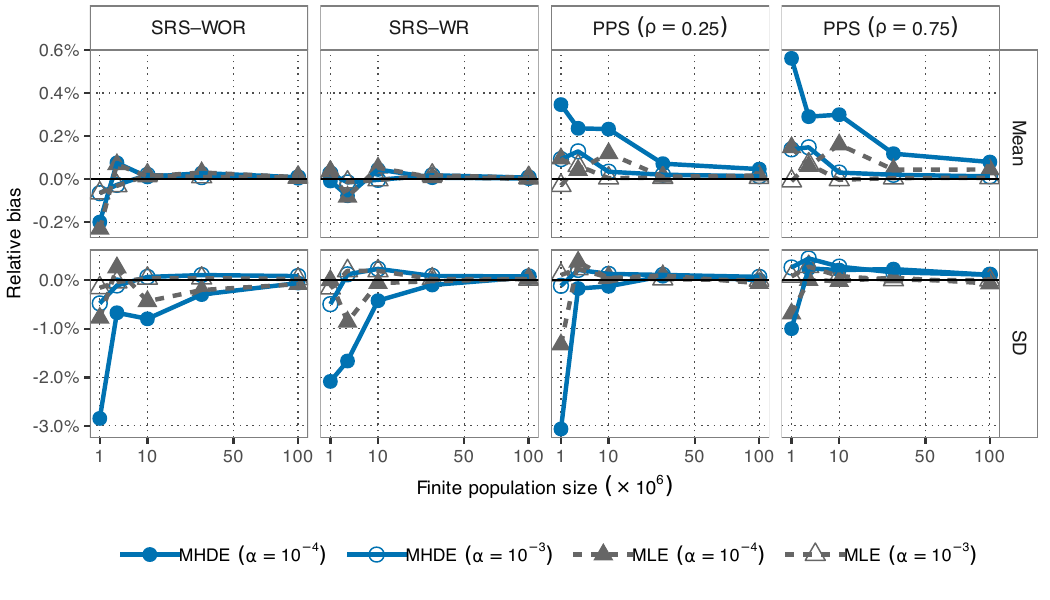}
    \caption{%
      Relative bias of the MHDE (blue dots) and MLE (gray triangles) in a Lognormal superpopulation model using various sampling designs.
      The sample size in each simulation is determined by $n=\alpha N$, with $\alpha \in \{10^{-3},10^{-4}\}$.%
      }
    \label{fig:sim-lognormal-rel-bias}
\end{figure}

\begin{figure}[t]
    \centering
    \includegraphics[width=1\linewidth]{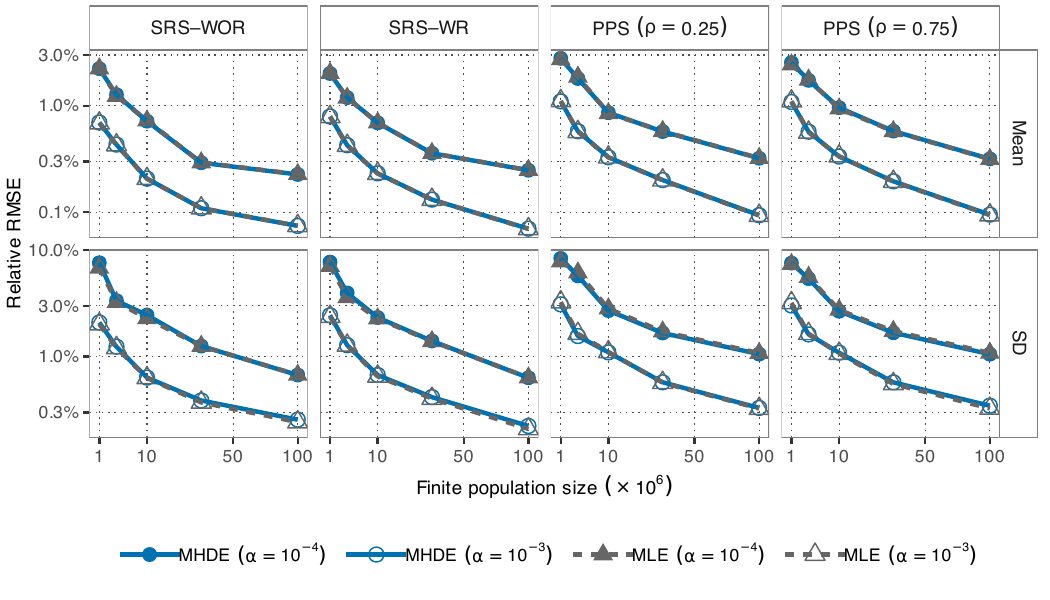}
    \caption{%
      Relative RMSE of the MHDE (blue dots) and MLE (gray triangles) in a Lognormal superpopulation model using various sampling designs.
      The sample size in each simulation is determined by $n=\alpha N$, with $\alpha \in \{10^{-3},10^{-4}\}$.%
      }
    \label{fig:sim-lognormal-rel-rmse}
\end{figure}

\begin{figure}[t]
    \centering
    \includegraphics[width=1\linewidth]{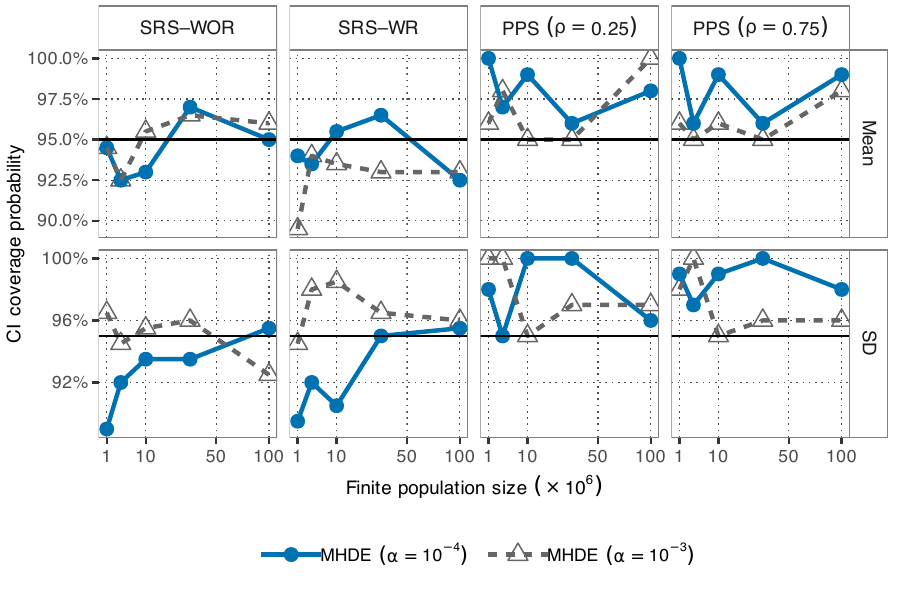}
    \caption{%
      Coverage probability of the 95\% confidence interval for the MHDE in a Lognormal superpopulation model as the finite population size $N$ increases using different sampling designs.
      The sample size in each simulation is determined by $n=\alpha N$, with $\alpha \in \{10^{-3},10^{-4}\}$.%
      }
    \label{fig:sim-lognormal-inference}
\end{figure}

\end{appendices}

\clearpage

\bibliography{references}

\end{document}